\documentclass[11pt,english,a4paper]{smfart}
\AtBeginDocument{\renewcommand*{\smfbyname}{}}
\usepackage[english]{babel}
\usepackage[T1]{fontenc}
\usepackage[utf8]{inputenc}
\usepackage{lmodern}
\usepackage[a4paper]{geometry}

\usepackage{amsmath}
\usepackage{amsfonts}
\usepackage{amssymb}
\usepackage{amsthm}

\usepackage{mathtools}
\usepackage{stmaryrd}
\usepackage{cases}
\usepackage{eqnarray}
\usepackage{nicematrix}
\usepackage{empheq}
\usepackage{csquotes}
\usepackage{parskip}

\usepackage{hyperref}
\hypersetup{colorlinks,linkcolor=blue,citecolor=blue,urlcolor=red}

\usepackage[backend=biber,style=numeric,maxbibnames=9,backref=true]{biblatex}

\usepackage{cleveref}
\renewcommand{\thesubsection}{\thesection.\Alph{subsection}}
\numberwithin{equation}{section}

\theoremstyle{plain}
\newtheorem{theo}{Theorem}[section]
\newtheorem{prop}[theo]{Proposition}

\newtheorem{cor}[theo]{Corollary}
\newtheorem{lem}[theo]{Lemma}

\theoremstyle{definition}
\newtheorem{defi}[theo]{Definition}

\theoremstyle{remark}
\newtheorem{rema}[theo]{Remark}

\newtheorem{fait}{Fact}

\NewDocumentCommand{\cj}{m}{\overline{#1}}

\NewDocumentCommand{\ps}{m}{\langle #1 \rangle}
\NewDocumentCommand{\id}{}{\ensuremath{\mathrm{Id}}}
\NewDocumentCommand{\oud}{}{\ensuremath{\mathbb{D}}}
\NewDocumentCommand{\cud}{}{\ensuremath{\overline{\mathbb{D}}}}
\NewDocumentCommand{\mpt}{m}{\ensuremath{\underline{#1}}}

\newcommand{\T}{\mathbb{T}}
\newcommand{\C}{\mathbb{C}}

\newenvironment{entry}
{\begin{list}{X}%
  {%
      \setlength{\labelwidth}{55pt}%
      \setlength{\leftmargin}{\labelwidth}
      \addtolength{\leftmargin}{\labelsep}%
   }%
}%
{\end{list}}   

\author[C. Badea]{Catalin Badea}
\address[C. Badea]{Univ. Lille, CNRS, UMR 8524 - Laboratoire Paul
Painlev\'e, F-59000 Lille, France}
\email{catalin.badea@univ-lille.fr}

\author[A. Renard]{Axel Renard}
\address[A. Renard]{Univ. Lille, CNRS, UMR 8524 - Laboratoire Paul
Painlev\'e, F-59000 Lille, France}
\email{axel.renard@univ-lille.fr}

\thanks{The authors acknowledge support from the Labex CEMPI (ANR-11-LABX-0007-01).}

\title[Schwarz-Pick type inequalities]{Schwarz-Pick type inequalities from an operator theoretical point of view}

\begin{document}
\frontmatter

\begin{abstract}
    We use (versions of) the von Neumann inequality for Hilbert space contractions to prove several Schwarz-Pick type inequalities. Specifically, we derive an alternate proof for a multi-point Schwarz-Pick inequality by Beardon and Minda, along with a generalized version for operators. Connections with model spaces and Peschl’s invariant derivatives are established. Finally, Schwarz-Pick inequalities for analytic functions on polydisks and for higher order derivatives are discussed. An enhanced version of the Schwarz-Pick lemma, using the notion of distinguished variety, is obtained for the bidisk.
\end{abstract}

\subjclass{47A25, 47A30, 47A08, 30A10, 30C80, 15A60}
\keywords{von Neumann inequality, Schwarz-Pick estimates, Beardon-Minda inequality, Peschl's invariant derivatives}


	\maketitle

\tableofcontents

\mainmatter

\section*{Nomenclature}

\begin{entry}

\item[$\oud$] denotes the open unit disk

\item[$\mathbb{T}$] denotes the unit circle,  $\T = \cud \backslash \oud$

\item[$\left\| \cdot \right\|_{\infty}$]  for $\mpt{\omega}=(\omega_1, \dots, \omega_n) \in \mathbb{C}^n$, we denote $\left\| \mpt{\omega}\right\|_{\infty} = \sup \{ |\omega_i| : 1 \leq i \leq n \}$

\item[$\mathcal{H}(\oud)$] is the set of functions that are holomorphic on $\oud$, $\mathcal{H}(\oud) = \mathcal{H}(\oud,\C)$ 

\item[$\mathcal{H}(\oud,\oud)$] denotes the set of functions in $\mathcal{H}(\oud)$ mapping $\oud$ to $\oud$

\item[$\mathcal{A}(\cud)$] is the disk algebra, \textit{i.e.} the set of functions that are holomorphic on $\oud$ and continuous on $\cud$

\item[$(z,w)$] denotes the complex pseudo-hyperbolic distance $(z,w) := \frac{z-w}{1-\overline{w}z}$

\item[$\rho(z,w)$] denotes the pseudo-hyperbolic distance $\rho(z,w):=|(z,w)|$

\item[$\mathrm{d}(z,w)$] denotes the hyperbolic distance $\mathrm{d}(z,w)=\tanh^{-1}(\rho(z,w))$

\item[$f^*(z,w)$] denotes the hyperbolic divided difference $f^*(z,w) : = \frac{(f(z),f(w))}{(z,w)}$

\item[$H^2(\oud)$] is the Hilbert-Hardy space of $\oud$

\item[$H^{\infty}(\oud)$] is the set of bounded holomorphic functions on $\oud$

\item[$\mathcal{B}(H,K)$] is the set of bounded linear operators from $H$ to $K$, where $H$ and $K$ are 
two complex Hilbert spaces. $\mathcal{B}(H)$ is a short for $\mathcal{B}(H,H)$

\item[$\|\cdot \|$] denotes the norm of an element of the Banach space under consideration. When $T\in \mathcal{B}(H,K)$, $\|T\|$ denotes the operator norm

\item[$T^*$] denotes the adjoint of $T$, where $T$ is a Hilbert space operator

\item[$D_T$] denotes the defect operator of a contraction $T\in\mathcal{B}(H)$, \text{i.e.} $\|T\|\le 1$. Thus $D_T = (\id - T^*T)^{1/2}$, where $\id$ is the identity operator

\item[$\sigma(T)$] denotes the spectrum of $T\in \mathcal{B}(H)$

\item[$r(T)$] denotes the spectral radius $r(T) = \sup\{|\lambda| : \lambda\in\sigma(T)\}$ of $T$

\item[$\text{Ker}(T)$] denotes the kernel of $T$

\item[$\text{Im}(T)$] denotes the range (image) of $T$


\item[$\llbracket 1 , n \rrbracket$] denotes the set of all integers $j$ with $1\le j \le n$.

\end{entry}

\section{Introduction}
\label{s:Introduction}

The Schwarz-Pick inequality, an invariant form of the Schwarz lemma, stands as a cornerstone in complex analysis. In geometric terms, it posits that a holomorphic map from the open unit disk into itself has the property of decreasing the distance between points in the hyperbolic metric.
Equivalently, if $f\in \mathcal{H}(\oud,\oud)$ and $\omega_1$ and $\omega_2$ are two points in $\oud$, then
\begin{equation}
\label{eq:Schw-Pick}
	\rho(f(\omega_1),f(\omega_2)) = \left| \frac{f(\omega_1) - f(\omega_2)}{1 - \overline{f(\omega_1)} f(\omega_2)} \right| \leq \left| \frac{\omega_1 - \omega_2}{1 - \overline{\omega_1}\omega_2} \right| = \rho(\omega_1,\omega_2).
\end{equation}
Inequality \eqref{eq:Schw-Pick} is strict for $\omega_1\neq \omega_2$ unless $f$ is a conformal automorphism of the unit disk.
Moreover, 
\begin{equation}
\label{eq:Schw-Pick-der}
	 \frac{\left|f'(\omega)\right|}{1 - \left|f(\omega)\right|^2} \le \frac{1}{1 - \left|\omega\right|^2}.
\end{equation}

The Schwarz-Pick inequalities \eqref{eq:Schw-Pick} and \eqref{eq:Schw-Pick-der} have been extended in various ways by many authors. A thorough overview of some of these advancements is provided in the comprehensive survey \cite{history}.  For the present study the following `three points' version of (\ref{eq:Schw-Pick}) by Beardon and Minda \cite{beardonMultipointSchwarzPickLemma2004} is pivotal. It involves the notion of hyperbolic divided difference $f^*(z,w)$ and states that if $f\in \mathcal{H}(\oud,\oud)$, then 
\begin{equation} \label{eq:BM1}
		\rho\left(f^*(\omega_1, \omega_2), f^*(\omega_3, \omega_2)\right) \leq \rho(\omega_1, \omega_3)
		\end{equation}
for three pairwise distinct points $\omega_1$, $\omega_2$ and $\omega_3$ in the unit disk. The Beardon-Minda inequality unifies in an elegant way many improvements of \eqref{eq:Schw-Pick}. An analogous theorem with more than three points has been proved in \cite{baribeauHyperbolicDividedDifferences2009}, where Baribeau, Rivard and Wegert also used iterated (hyperbolic) divided differences to give simpler conditions for the $n$ points Nevanlinna–Pick interpolation problem. We refer also to \cite{Abate,RivardPAMS,RivardCAOT} for related contributions.

Various operator-theoretical interpretations of the Schwarz-Pick inequality are possible. The most well-known interpretation, due to Sarason~\cite{sarasonGeneralizedInterpolationSpinfty1967}, exploits the equivalence of the Schwarz-Pick inequality with the Nevanlinna-Pick interpolation problem for two points. Consequently, the Schwarz-Pick inequality possesses an operator-theoretical significance concerning norm-preserving lifting of some operators that act on specific subspaces of the Hardy space $H^2(\oud)$. Additional generalizations can be derived using the commutant lifting theorem of Sz.-Nagy and Foias (see for instance \autocite{FoiasFrazho}). 
It is noteworthy that the commutant lifting theorem is equivalent with Ando's dilation theorem (\autocite{NikVas}). 
Further operator-theoretical interpretations of the Schwarz-Pick inequality have been explored in \autocite{AR,ADR,jocicNoncommutativeSchwarzLemma2022, kneseSchwarzLemmaPolydisk2007a, McC1}.

The starting point of this note was the natural question of looking for an operator-theoretical interpretation of the Beardon-Minda inequality. Notice that the Schwarz-Pick inequality can be obtained as a particular case of the von Neumann inequality for Hilbert space contractions. The von Neumann inequality states that if $T \in \mathcal{B}(H)$ is a bounded linear operator acting on a complex Hilbert space $H$ with $\|T\| \le 1$ and $f$ is a polynomial, then 
\begin{equation}
\label{eq:vN}
\left\| f(T)\right\| \le \sup \{ |f(z)| : |z| \le 1\} .
\end{equation}
This inequality extends to functions $f$ in the disk algebra $\mathcal{A}(\cud)$.  The inequality (\ref{eq:Schw-Pick}) is obtained (\autocite[p.17]{rosenblumHardyClassesOperator1997}, \autocite[Exercices 2.17-2.18]{paulsenCompletelyBoundedMaps2002}) when applying von Neumann's inequality (\ref{eq:vN}) to a polynomial (or an element in the disk algebra) $f$ and a specific $2 \times 2$ matrix
acting on the $2$-dimensional Hilbert space $\C^2$.  Then, an approximation argument gives (\ref{eq:Schw-Pick}) for $f\in \mathcal{H}(\oud,\oud)$. A similar derivation by Agler of the Schwarz-Pick inequality (\ref{eq:Schw-Pick}) is presented in \cite{Agler:Invent} and \cite[Chapter 8]{aglerOperatorAnalysisHilbert2020}, where this concept is ingeniously applied in to establish an operator-theoretical proof of Lempert's theorem, demonstrating the equality of the Carath\'{e}odory and Kobayashi metrics on convex domains. In this case, von Neumann's inequality is substituted with the notion of spectral set.

The primary objective of this manuscript is to leverage (versions of) the von Neumann inequality to derive Schwarz-Pick type inequalities.  
Notably, when applying the von Neumann inequality to a remarkable $3 \times 3$ matrix (the matrix of the model operator in the Takenaka-Malmquist basis), the Beardon-Minda three-point Schwarz-Pick inequality is obtained. A similar proof is given for a Beardon-Minda type inequality for derivatives and operator versions of the Schwarz-Pick and Beardon-Minda type inequalities are obtained. We also consider some Schwarz-Pick related inequalities for higher derivatives.
Additionally, we delve into the multivariable case, employing von Neumann's inequality on $n$-tuples of mutually commuting $2 \times 2$ or $3\times 3$ matrices to prove Schwarz-Pick type inequalities for the polydisk. In the case of the bidisk an improvement can be given using the notion of distinguished variety and the refined version of Ando's inequality by Agler and McCarthy \cite{AM:acta}. The reader is welcomed to notice that the multipoint Schwarz-Pick inequality of \cite{baribeauHyperbolicDividedDifferences2009} can also be derived as a consequence of the von Neumann inequality. However, explicit computations with matrices become more intricate.

{\bf Outline.} The manuscript is organized as follows. In the next section we use a theorem going back to Parrott to obtain criteria for scalar and operator $3\times 3$ matrices to have (Hilbertian) operator norm no greater than one. This is applied to a specific matrix to obtain an alternate proof of the Beardon-Minda inequality. The significance of this specific matrix with model spaces is highlighted, and a discussion concerning the equality case in \eqref{eq:Schw-Pick} is given. Also, a Beardon-Minda type inequality for derivatives, originally proved by Goluzin and Yamashita~\cite{yamashitaPickVersionSchwarz1994}, is obtained as a consequence of the von Neumann inequality. This inequality can be rephrased in terms of Peschl’s invariant derivatives. 

In Section~\ref{sect:3} we prove several operator versions of the Schwarz-Pick inequality and of the Beardon-Minda inequality. The Sylvester (operator) equation $AX-XB = Y$ plays an important role in the proofs. 

In the next section we consider the case of the polydisk. We give operator theoretical proofs of the analogues of \eqref{eq:Schw-Pick} and \eqref{eq:Schw-Pick-der} for the polydisk and discuss the Peschl’s invariant derivatives in several variables. The proofs uses a result of Knese \cite{kneseSchwarzLemmaPolydisk2007a} that the (polydisk) von Neumann’s inequality holds for $n$-tuples of $3 \times 3$ commuting contractive matrices. In the case of the bidisk we use a result of Agler and McCarthy~\cite{AM:acta} to obtained an enhanced version of the Schwarz-Pick inequality. 

In Section~\ref{sect:5} we give some Schwarz-Pick related inequalities for higher derivatives. An improvement of a classical result of F.~Wiener is proved. 

Parrott's theorem, which was essential in our proofs, is revisited in the Appendix. We hope that this will prove beneficial for readers interested in Schwarz-Pick inequalities who may not be extensively acquainted with operator theory.

\section{A three points Schwarz-Pick lemma}

In view of the preceding discussion, it is a natural question to apply the von Neumann inequality to $3 \times 3$ matrices.

\subsection{Contractive three by three matrices}
First, we need an explicit criterion to determine whenever a $3 \times 3$-upper triangular matrix is a contraction. Note that in view of Schur's decomposition theorem -- which states that every square matrix is unitary equivalent to an upper-triangular matrix -- it makes sense to restrict ourselves to that case. In essence, following the approach used for $2\times 2$ matrices, one can calculate the operator norm of a $3\times 3$ matrix acting on the Euclidean space $\mathbb{C}^3$ using the formula: 
$$ \|T\|^2 = \|T^*T\| = r(T^*T) = \sup\{|\lambda| : \text{det}(T^*T - \lambda \id) = 0\}.$$
This computation of the operator norm $\|T\|$, the largest singular value of $T$, leads to an equation of degree $3$. However, the criterion derived from this observation holds limited practical interest. An alternative approach to obtain such a criterion uses the Schur parameters (cf. \cite{constantinescuSchurParametersFactorization1996}). 
Adapting the argument in \autocite[Lemma 2.7]{gupta}, we follow here a different approach, based on a result about completion of matrices going back to Parrott (see \autocite{parrottQuotientNormSz1978, FoiasFrazho}, \autocite[Theorem 12.22]{youngIntroductionHilbertSpace1988} and \autocite{ArseneGheondea, davisNormpreservingDilationsTheir1982}). 

\begin{theo}[Parrott]\label{parrott}
	Let $H_1, H_2, K_1, K_2$ be Hilbert spaces, and suppose that the operators $\begin{bmatrix}
		A \\ C
	\end{bmatrix} \in \mathcal{B}(H_1, K_1 \oplus K_2)$ and $\begin{bmatrix}
		C & D
	\end{bmatrix} \in \mathcal{B}(H_1 \oplus H_2, K_2)$ are contractions.
	Then, $T= \begin{bmatrix}
		A & B \\ C & D
	\end{bmatrix} : H_1  \oplus H_2 \to K_1 \oplus K_2$ is a contraction if and only if there exists a contraction $W \in \mathcal{B}(H_2,K_1)$ such that $B=D_{Z^*}WD_Y - Z C^*Y$, where $Z \in \mathcal{B}(H_1, K_1)$ and $Y \in \mathcal{B}(H_2,K_2)$ are contractions such that $D=D_{C^*}Y$ and $A=ZD_C$. 
 
 Moreover, 
	\begin{enumerate}
		\item $Y$ and $Z$ can be chosen to be (respectively) $Y_0$ and $Z_0$, the solutions of minimal operator norm among all solutions of the operator equations $D=D_{C^*}Y$ and  $A=ZD_C$;
		\item If $T$ is a contraction, there exists a unique contraction $W_0$ such that \[B=D_{Z_0^*}W_0D_{Y_0} - Z_0 C^*Y_0 \; \text{ and } \; \operatorname{Im}\left(D_{Z_0^*}\right)^{\perp} \subset \operatorname{Ker}(W_0^*). \]
  This operator satisfies \[ \|W_0\| = \inf\{ \, \|W\| \, : \, B=D_{Z_0^*}WD_{Y_0} - Z_0 C^*Y_0 \} .\]
	\end{enumerate} 
\end{theo}
We shall call $Y_0$ and $Z_0$ the \emph{minimal solutions} and we shall refer to $W_0$ as the \emph{minimal solution} of the equation \[B=D_{Z_0^*}WD_{Y_0} - Z_0 C^*Y_0.\] A further discussion is given in the Appendix (Section~\ref{sect:6}). In particular, $T= \begin{bmatrix}
		A_1 & B \\ 0 & A_2
	\end{bmatrix}$
is a contraction if and only if $B = (I-A_1A_1^{*})^{1/2}W(I-A_2^{*}A_2)^{1/2}$ for a certain contraction $W$ and the scalar matrix 
$$ T= \begin{bmatrix}
		\omega_1 & a \\ 0 & \omega_2
	\end{bmatrix} $$
has (Euclidean) norm no greater than one if and only if  $|\omega_1| \le 1$, $|\omega_2| \le 1$ and
 $|a| \le \sqrt{1-|\omega_1|^2}\sqrt{1-|\omega_2|^2}$.

 The following result provides a criterion for determining whether a $3 \times 3$ operator matrix is a contraction, when the central entry of the matrix, $W_2$, is a strict contraction.

\begin{theo} \label{theo:OperBM}
	Let $H_1, H_2, H_3$ be three Hilbert spaces. Let $W_i \in \mathcal{B}(H_i)$, $1 \leq i \leq 3$, be three contractions and denote
 $$T=\begin{bmatrix}
		W_1 & A_1 & B \\
		0 & W_2 & A_2 \\
		0 & 0 & W_3
	\end{bmatrix} \in 
 \mathcal{B}(H_1 \oplus H_2 \oplus H_3).$$ 
 Assume that $||W_2|| < 1$. Then, $T$ is a contraction if and only if there exist three contractions $V_1 \in \mathcal{B}(H_2,H_1), V_2 \in \mathcal{B}(H_3,H_2), V_3 \in \mathcal{B}(H_3,H_1)$ such that :
	\begin{numcases}{}
		A_1 = D_{W_1^*}V_1D_{W_2}, \label{parott-3-3-cond1} \\
		A_2 = D_{W_2^*}V_2D_{W_3}, \label{parott-3-3-cond2} \\
		B = \left[D_{W_1^*}(\id-V_1V_1^*)D_{W_1^*} \right]^{1/2} V_3 \left[ D_{W_3}(\id - V_2^*V_2)D_{W_3} \right]^{1/2} \notag \\
  - D_{W_1^*}V_1W_2^*V_2D_{W_3}. \label{parott-3-3-cond3}
	\end{numcases} 
\end{theo}

\begin{proof}
		First, if $T$ is a contraction, then $\begin{bmatrix}W_1 & A_1 \\ 0 & W_2 \end{bmatrix}$ and $\begin{bmatrix}W_2 & A_2 \\ 0 & W_3 \end{bmatrix}$ are also contractions, as they are compressions of $T$. Then Parrott's theorem implies that \eqref{parott-3-3-cond1} and \eqref{parott-3-3-cond2} are satisfied. In the following we assume that \eqref{parott-3-3-cond1} and \eqref{parott-3-3-cond2} are true.
	
	Now, denote
 $$A=\begin{bmatrix}
		W_1 & A_1
	\end{bmatrix}, \quad C=\begin{bmatrix}
		0 & W_2 \\ 0 & 0
	\end{bmatrix} \quad \textrm{and} \quad D=\begin{bmatrix}
		A_2 \\ W_3
	\end{bmatrix}.$$
		By Parrott's theorem, $T$ is a contraction if and only if :
	\begin{equation} \label{parrott-appli-a-3}
		B = (\mathrm{Id}-ZZ^*)^{1/2}V_3(\mathrm{Id}-Y^*Y)^{1/2}-ZC^*Y, \end{equation}
  for an arbitrary 
  contraction $V_3 \in \mathcal{B}(H_3,H_1)$. Here $Y$ and $Z$ are contractions such that $D=(\mathrm{Id}-CC^*)^{1/2}Y$ and $A=Z(\mathrm{Id}-C^*C)^{1/2}$, the existence of which is
 ensured by Parrott's theorem for column (respectively row) matrix-operators. Indeed, $\begin{bmatrix}
		A \\ C
	\end{bmatrix}$ and $\begin{bmatrix}
		C & D
	\end{bmatrix}$ are contractions. 
	We have $\id-CC^* = \begin{bmatrix}\id - W_2W_2^* & 0 \\
		0 & \id \end{bmatrix}$ and $\id-C^*C = \begin{bmatrix}
		\id & 0 \\
		0 & \id-W_2^*W_2
	\end{bmatrix}$. 
	Since $\|W_2\| < 1$, these operators are invertible. Thus, we get \begin{align*}
	    & Z=AD_C^{-1}=\begin{bmatrix}
	W_1 & A_1 D_{W_2}^{-1}
	\end{bmatrix} = \begin{bmatrix} W_1 & D_{W_1^*}V_1\end{bmatrix}, \\ & Y=D_{C^*}^{-1}D=\begin{bmatrix}
	D_{W_2^*}^{-1}A_2 \\ W_3
	\end{bmatrix} = \begin{bmatrix}
	V_2 D_{W_3} \\ W_3
	\end{bmatrix}.
	\end{align*}
	
	It follows that $D_{Z^*}=\left[ D_{W_1^*}(\id - V_1V_1^*)D_{W_1^*}\right]^{1/2}$, $D_Y=\left[ D_{W_3}(\id - V_2^*V_2)D_{W_3} \right]^{1/2}$ and\\ $ZC^*Y=D_{W_1^*}V_1W_2^*V_2D_{W_3}$.
	Therefore, \eqref{parrott-appli-a-3} is equivalent to \eqref{parott-3-3-cond3}.
\end{proof}
We obtain the following general criterion in the scalar case.
	\begin{theo}\label{theo:gupta}
		Let $\omega_1, \omega_2, \omega_3 \in \cud$. Then, $T = \begin{pmatrix}\omega_1 & \alpha_1 & \beta \\ 0 & \omega_2 & \alpha_2 \\ 0 & 0 & \omega_3\end{pmatrix}$ is a contraction when acting on the Hilbert space $\C^3$ if and only if 
		\begin{numcases}{}
			|\omega_2| < 1, \\
			\left|\alpha_i\right|^2 \leq (1-|\omega_i|^2)(1-|\omega_{i+1}|^2)\label{cond-1}, \: i=1,2, \\
			\left|\beta (1-|\omega_2|^2)+\alpha_1\alpha_2\overline{\omega_2}\right|^2 \le \notag\\
   \left[(1-|\omega_1|^2)(1-|\omega_2|^2) - |\alpha_1|^2\right] \cdot \left[(1-|\omega_2|^2)(1-|\omega_3|^2) - |\alpha_2|^2 \right]\label{cond-3}
		\end{numcases}
  or 
  \begin{numcases}{}
|\omega_2| =1, \\
			\alpha_i =0, \: i=1,2, \\
			|\beta|^2  \leq (1-|\omega_1|^2)(1-|\omega_3|^2).\label{eq:26}
  \end{numcases}
	\end{theo}
	\begin{proof}
		As in the proof of Theorem~\ref{theo:OperBM}, if $T$ is a contraction, then the two dimensional compressions $\begin{bmatrix}\omega_1 & \alpha_1 \\ 0 & \omega_2 \end{bmatrix}$ and $\begin{bmatrix}\omega_2 & \alpha_2 \\ 0 & \omega_3 \end{bmatrix}$ are also contractions. Thus \eqref{cond-1} is satisfied, and it will be assumed from now on. 
  Note that if $|\omega_2| =1$, this implies that $\alpha_1 = \alpha_2 =0$.
		
		We use similar notation as in the proof of Theorem~\ref{theo:OperBM}, with 
  $$A=\begin{bmatrix}
			\omega_1 & \alpha_1
		\end{bmatrix}, \quad B = \begin{bmatrix} \beta \end{bmatrix}, \quad C=\begin{bmatrix}
			0 & \omega_2 \\ 0 & 0
		\end{bmatrix} \quad \textrm{and} \quad D=\begin{bmatrix}
			\alpha_2 \\ \omega_3
		\end{bmatrix}.$$
		
		By Theorem~\ref{parrott}, $T$ is a contraction if and only if :
		\begin{equation}{} 
			B = (\mathrm{Id}-ZZ^*)^{1/2}V(\mathrm{Id}-Y^*Y)^{1/2}-ZC^*Y, \text{ for some contraction } V,  \label{par-3} \end{equation}
		
		where $Y$ and $Z$ are contractions such that $D=(\mathrm{Id}-CC^*)^{1/2}Y$ and $A=Z(\mathrm{Id}-C^*C)^{1/2}$.
		
		We have 
		$$\mathrm{Id} - CC^* = \begin{bmatrix}1 - |\omega_2|^2 & 0 \\
			0 & 1 \end{bmatrix} $$
		and
		$$\mathrm{Id}-C^*C = \begin{bmatrix}
			1 & 0 \\
			0 & 1-|\omega_2|^2
		\end{bmatrix}.$$

 \emph{First case}. Assume first that $|\omega_2|<1$. Then we can apply Theorem~\ref{theo:OperBM}. An easy computation shows that \[Y= \left(\mathrm{Id} - CC^* \right)^{-1/2} D = \begin{bmatrix}
				\frac{\alpha_2}{\sqrt{1-|\omega_2|^2}} \\ \omega_3
			\end{bmatrix}\] and \[ Z= A \left(\mathrm{Id} - C^* C \right)^{-1/2} = \begin{bmatrix}
				\omega_1 & \frac{\alpha_1}{\sqrt{1-|\omega_2|^2}}\end{bmatrix}.\]
			Thus, $T$ is a contraction if and only if \eqref{par-3} is satisfied, that is
   $$\beta + \frac{\alpha_1 \alpha_2 \overline{\omega_2}}{1 - |\omega_2|^2} = \left(1- |\omega_1|^2 - \frac{|\alpha_1|^2}{1 - |\omega_2|^2}\right)^{1/2} V \left(1- |\omega_3|^2 - \frac{|\alpha_2|^2}{1 - |\omega_2|^2}\right)^{1/2}$$
   for some contraction $V$. This holds if and only if 
   $$\left| \left(1- |\omega_1|^2 - \frac{|\alpha_1|^2}{1 - |\omega_2|^2}\right)^{-1/2} \left(\beta + \frac{\alpha_1 \alpha_2 \overline{\omega_2}}{1 - |\omega_2|^2} \right) \left(1- |\omega_3|^2 - \frac{|\alpha_2|^2}{1 - |\omega_2|^2}\right)^{-1/2} \right| \leq 1.$$
   In can be easily shown that this is equivalent to
   the condition \eqref{cond-3}.

	 \emph{Second case}. Assume now that $|\omega_2|=1$. Let $Y=\begin{bmatrix}
				y_1 \\ y_2
			\end{bmatrix}$ and $Z= \begin{bmatrix}
				z_1 & z_2
			\end{bmatrix}$.
   As $D= (\mathrm{Id}-CC^*)^{1/2}Y$, we get $D^* = Y^* (\mathrm{Id}-CC^*)^{1/2}$. This holds if and only if $y_2=\omega_3$.
			As $Y^*$ can be chosen to be $0$ on $\text{Im}(D_C^*)^{\perp}$ (see the Appendix), we have $y_1=0$.
			
			Similarly, $A=Z (\mathrm{Id}-C^* C)^{1/2}$ holds if and only if $z_1=\omega_1$ and, as before, we can choose $z_2=0$. We have $ZC^* Y = 0$. Therefore, $T$ is a contraction if and only if $|\beta|^2 \leq (1-|\omega_3|^2)(1-|\omega_1|^2)$. This is equivalent to the condition \eqref{eq:26}.
		\end{proof}

	\subsection{An operator-theoretical proof of Beardon-Minda's inequality}
	We refer to \cite[Chapter 22]{Simon} for the definition and basic properties of \emph{divided differences} of $n+1$ (not necessarily distinct) points. We just recall here that for pairwise distinct points $z_0, z_1, \cdots , z_n \in \mathbb{C}$, the divided differences of $f$ at points $z_0, z_1, \cdots , z_n$ satisfy $[f(z_k)] = f(z_k)$ and the recurrence relation
  $$\left[f(z_k), \cdots , f(z_{k+j})\right] = \frac{\left[f(z_{k+1}), \cdots , f(z_{k+j})\right] - \left[f(z_{k}), \cdots , f(z_{k+j-1})\right] }{z_{k+j} - z_k},$$ for $ 0 \leq k \leq j \leq n$.
	
	We also recall to the reader the following notation.
		\begin{defi}
		Let $z,w \in \oud$ and $f \in \mathcal{H}(\oud,\oud)$. We define:
		\begin{enumerate}
			\item The complex pseudo-hyperbolic distance $(z,w) := \frac{z-w}{1-\overline{w}z}$;
			\item The pseudo-hyperblic distance $\rho(z,w):=|(z,w)|$;
			\item The hyperbolic distance $\mathrm{d}(z,w)=\tanh^{-1}(\rho(z,w))$;
			\item The hyperbolic divided difference $f^*(z,w) : = \frac{(f(z),f(w))}{(z,w)}$.
		\end{enumerate}
	\end{defi}	

We provide an operator-theoretic proof of the following result established by Beardon and Minda~\autocite{beardonMultipointSchwarzPickLemma2004}. 
	
	\begin{theo}[Beardon-Minda \autocite{beardonMultipointSchwarzPickLemma2004}]\label{beardon-minda}
		Let $f \in \mathcal{H}(\oud,\oud)$ 
  and let $\omega_1$, $\omega_2$ and $\omega_3$ be pairwise distinct points in $\oud$. Then,
		\begin{equation} \label{eq:BM}
		\mathrm{d}\left(f^*(\omega_1, \omega_2), f^*(\omega_3, \omega_2)\right) \leq \mathrm{d}(\omega_1, \omega_3).
		\end{equation}
	\end{theo}
The proof in~\autocite{beardonMultipointSchwarzPickLemma2004} requires an assumption that $f$ is not a conformal automorphism of the unit disk. Such an assumption is unnecessary in the subsequent proof.
	\begin{proof}[Proof of Theorem~\ref{beardon-minda}]
		First, let us notice that Beardon-Minda's inequality \eqref{eq:BM} is equivalent to
		\begin{equation}\label{bmbis}
		\rho(f^*(\omega_1, \omega_2),f^*(\omega_3, \omega_2)) = \left| \frac{f^*(\omega_1, \omega_2) - f^*(\omega_3, \omega_2)}{1 - \overline{f^*(\omega_3, \omega_2)} f^*(\omega_1, \omega_2)} \right| \leq \left| \frac{\omega_1 - \omega_3}{1 - \overline{\omega_3}\omega_1} \right|.
		\end{equation}
		For $z, \omega \in \oud$, we have
		\begin{equation}\label{diff-div-comp}
			f^*(z, \omega) = \frac{f(z) - f(\omega)}{z- \omega} \cdot \frac{1-\overline{\omega}z}{1- \overline{f(\omega)} f(z)}.
		\end{equation}
		
		We also record the following important identity, valid for $u,v \in \mathbb{C}$. We have
		\begin{equation} \label{eq:important}
		    S_{u,v} := (1-|u|^2)(1-|v|^2) = |1- \overline{u}v |^2 - |u-v|^2.
		\end{equation}
		
		Now, let $\omega_1, \omega_2, \omega_3 \in \oud$, with $\omega_i \neq \omega_j$ ($i \neq j$), and consider
  $$T=\begin{pmatrix}\omega_1 & \alpha_1 & \beta \\ 0 & \omega_2 & \alpha_2 \\ 0 & 0 & \omega_3\end{pmatrix},$$ 
  with 
$$   \alpha_i=\sqrt{1-|\omega_i|^2}\sqrt{(1-|\omega_{i+1}|^2}, \qquad i=1,2, $$
   and
   $$
  \beta=\frac{- \overline{\omega_2} \alpha_1 \alpha_2}{1-|\omega_2|^2} = -\overline{\omega_2} \sqrt{1-|\omega_1|^2}\sqrt{1-|\omega_3|^2}. $$ 
    By Theorem \ref{theo:gupta}, $T$ is a contraction. Assume first that $f \in \mathcal{A}(\cud)$ and $\|f\|_{\infty} \le 1$. Then the matrix representation of $f(T)$ can be expressed in terms of first order and second order divided differences as follows:
		\[
		f(T) = \begin{pmatrix}
			f(\omega_1) & \alpha_1[f(\omega_1), f(\omega_2)] & \beta[f(\omega_1), f(\omega_3)] + \alpha_1 \alpha_2 [f(\omega_1), f(\omega_2), f(\omega_3)] \\
			0 & f(\omega_2) & \alpha_2 [f(\omega_2), f (\omega_3)] \\
			0 & 0 & f(\omega_3)
		\end{pmatrix}.
		\]
		This can be verified directly by some direct computations for monomials and polynomials. The same formula extends to functions in the disk algebra $\mathcal{A}(\cud)$. Assume that $f(\omega_i) \neq f(\omega_j)$ whenever $i \neq j$ (otherwise, there is nothing to prove).
		As $T$ is a contraction and $\|f\|_{\infty} \le 1$, by von Neumann's inequality the operator $f(T)$ is also a contraction.

		Introducing the notation \begin{align*}
		    & \widetilde{\alpha_i} = \alpha_i[f(\omega_i), f(\omega_{i+1})], \qquad i=1,2, \\ 
      & \widetilde{\beta}=\beta[f(\omega_1), f(\omega_3)] + \alpha_1 \alpha_2 [f(\omega_1), f(\omega_2), f(\omega_3)],
		\end{align*}
		
	by Theorem~\ref{theo:gupta}, we have :
		
		\[
		\left|\widetilde{\beta} \left(1-|f(\omega_2)|^2\right)+ \widetilde{\alpha_1} \widetilde{\alpha_2} \overline{f(\omega_2)}\right|^2 \leq \left[S_{f(\omega_1), f(\omega_2)} - |\widetilde{\alpha_1}|^2\right] \times
		\left[S_{f(\omega_2), f(\omega_3)} - |\widetilde{\alpha_2}|^2 \right].
		\]
		
		If we multiply each side of this inequality by $|\omega_1 - \omega_3|^2 $, we get
		\begin{equation}\label{ineq_mat_mod}
			S_{\omega_1,\omega_3}\left|\left(1-|f(\omega_2)|^2\right)A + B \right|^2 \leq |\omega_1 - \omega_3|^2 C_1 C_3 ,
		\end{equation}
		where 
		\begin{align*} 
			A & :=- \overline{\omega_2} \left(f(\omega_1) - f(\omega_3) \right) + \left(1-|\omega_2|^2\right) \left( \frac{f(\omega_1)-f(\omega_2)}{\omega_1-\omega_2}-\frac{f(\omega_2)-f(\omega_3)}{\omega_2-\omega_3}\right) ,  \\ 
			B & :=\overline{f(\omega_2)}(1-|\omega_2|^2)(\omega_1-\omega_3)\cdot \frac{(f(\omega_1)-f(\omega_2))(f(\omega_2)-f(\omega_3))}{(\omega_1-\omega_2)(\omega_2-\omega_3)} ,\\
			C_i & := S_{f(\omega_i), f(\omega_2)} - S_{\omega_i, \omega_2}\left|\frac{f(\omega_i)-f(\omega_2)}{\omega_i -\omega_2}\right|^2, \qquad i=1,3.
		\end{align*}
		
		We want to prove that \eqref{ineq_mat_mod} is equivalent to \eqref{bmbis}. The calculations are somewhat laborious; the key idea is to use \eqref{diff-div-comp} to make hyperbolic divided differences appear each time we see an expression of the form $f(z)-f(\omega)$. We provide additional details to assist the reader.
		
		First of all, we have :
		\begin{align*}
			C_i & = S_{f(\omega_i), f(\omega_2)} - S_{\omega_i,\omega_2} \left| f^*(\omega_i,\omega_2)\right|^2 \times \left|\frac{1-\overline{f(\omega_2)}f(\omega_i)}{1-\overline{\omega_2}\omega_i}\right|^2 \\
			& \begin{multlined}
				= \left|1-\overline{f(\omega_2)}f(\omega_i)\right|^2 - \left|f^*(\omega_i,\omega_2)\right|^2 \times \left|\omega_i -\omega_2 \right|^2 \times \left|\frac{1-\overline{f(\omega_2)}f(\omega_i)}{1-\overline{\omega_2}\omega_i}\right|^2  \\- \left| f^*(\omega_i,\omega_2)\right|^2 \times \left|1-\overline{f(\omega_2)}f(\omega_i)\right|^2  + \left|f^*(\omega_i,\omega_2)\right|^2 \times \left|\omega_i -\omega_2 \right|^2 \times \left|\frac{1-\overline{f(\omega_2)}f(\omega_i)}{1-\overline{\omega_2}\omega_i}\right|^2
			\end{multlined} \\
			& = \left|1-\overline{f(\omega_2)}f(\omega_i)\right|^2 \left(1-\left|f^*(\omega_i,\omega_2)\right|^2\right).
		\end{align*}
		
		Thus, we have 
  $$C_1C_3 = \left|1-\overline{f(\omega_2)}f(\omega_1)\right|^2 \times \left|1-\overline{f(\omega_2)}f(\omega_3)\right|^2 \times S_{f^*(\omega_1,\omega_2), f^*(\omega_3,\omega_2)}.$$
		
		Now, let us deal with the first member of the inequality. We have
		\begin{align*}
			A & = f^*(\omega_1,\omega_2)\left(1-\overline{f(\omega_2)}f(\omega_1)\right) - f^*(\omega_3,\omega_2)\left(1-\overline{f(\omega_2)}f(\omega_3)\right) \\
			& = \left( f^*(\omega_1,\omega_2) - f^*(\omega_3,\omega_2)\right) \left(1-\overline{f(\omega_2)}f(\omega_1)\right)\left(1-\overline{f(\omega_2)}f(\omega_3)\right) + \overline{f(\omega_2)} D,
		\end{align*}
		
		where $D := f^*(\omega_1,\omega_2)f(\omega_3)\left(1-\overline{f(\omega_2)}f(\omega_1)\right) - f^*(\omega_3,\omega_2)f(\omega_1)\left(1-\overline{f(\omega_2)}f(\omega_3)\right)$.
		
		This term can be written as follows, where we make appear the differences $f(\omega_3)-f(\omega_2)$ and $f(\omega_1)-f(\omega_2)$:
		\begin{align*}
			& \begin{multlined} D =
				f^*(\omega_1,\omega_2)\left(f(\omega_3)-f(\omega_2)\right)\left( 1-\overline{f(\omega_2)}f(\omega_1)\right) \\ - f^*(\omega_3,\omega_2)\left(f(\omega_1) -f(\omega_2)\right)\left( 1-\overline{f(\omega_2)}f(\omega_3)\right) + f(\omega_2)f^*(\omega_1,\omega_2)\left(1-\overline{f(\omega_2)}f(\omega_1)\right) \\ - f(\omega_2)f^*(\omega_3,\omega_2)\left(1-\overline{f(\omega_2)}f(\omega_3)\right).
			\end{multlined}
   \end{align*}
   We obtain 
   \begin{align*}
    D & = \begin{multlined}[t]
				 f^*(\omega_1,\omega_2)f^*(\omega_3,\omega_2)\left(\frac{ 1-\overline{f(\omega_2)}f(\omega_3)}{1-\overline{\omega_2} \omega_3}\right) (\omega_3-\omega_2) \left(1-\overline{f(\omega_2)}f(\omega_1)\right)  \\
				- f^*(\omega_3,\omega_2)f^*(\omega_1,\omega_2)\left(\frac{ 1-\overline{f(\omega_2)}f(\omega_1)}{1-\overline{\omega_2} \omega_1}\right) (\omega_1-\omega_2) \left(1-\overline{f(\omega_2)}f(\omega_3)\right) + f(\omega_2) A 
			\end{multlined} \\
			& = - (1-|\omega_2|^2)(\omega_1-\omega_3)\cdot \frac{(f(\omega_1)-f(\omega_2))(f(\omega_2)-f(\omega_3))}{(\omega_1-\omega_2)(\omega_2-\omega_3)} + f(\omega_2) A.
		\end{align*}
		
		Hence, we get \[ A= \left( f^*(\omega_1,\omega_2) - f^*(\omega_3,\omega_2)\right) \left(1-\overline{f(\omega_2)}f(\omega_1)\right)\left(1-\overline{f(\omega_2)}f(\omega_3)\right)-B+|f(\omega_2)|^2A .\] Therefore \[(1-|f(\omega_2)|^2)A + B = \left( f^*(\omega_1,\omega_2) - f^*(\omega_3,\omega_2)\right) \left(1-\overline{f(\omega_2)}f(\omega_1)\right)\left(1-\overline{f(\omega_2)}f(\omega_3)\right). \]
		
	Combining all of these elements, the inequality represented by \eqref{ineq_mat_mod} transforms into
		
		\[
		S_{\omega_1, \omega_3} \left| f^*(\omega_1,\omega_2) - f^*(\omega_3,\omega_2)\right|^2 \leq |\omega_1 - \omega_3|^2 \times S_{f^*(\omega_1,\omega_2), f^*(\omega_3,\omega_2)},
		\]
		
		which is equivalent to \eqref{bmbis}.
		
		Beardon-Minda's inequality is thus proved for $f \in \mathcal{A}(\cud)$. Now, for $f \in \mathcal{H}(\oud)$, we have $f_r : z \mapsto f(rz) \in \mathcal{A}(\cud)$, for every $r \in ]0,1[$. Based on the preceding information, it can be concluded that Beardon-Minda's inequality is satisfied by the functions $f_r$, for all $r \in ]0,1[$, so it is also by $f$, by letting $r \to 1^-$.
	\end{proof}
	The calculations in this proof can be somewhat simplified by assuming that $f(\omega_2)=0$ and composing with a Möbius transformation at the end. However, this approach leads to a loss of symmetry in the formulas.
	\subsection{Connecting with model spaces theory}
 In \autocite{beardonMultipointSchwarzPickLemma2004}, it is further proved that if $f$ does not represent an automorphism of the unit disk, equality holds in \Cref{beardon-minda} if and only if $f$ is a Blaschke product of degree no greater than $2$. This inference can also be derived through operator theory considerations.

To achieve this, we must introduce certain concepts from model space theory. For a comprehensive introduction to these notions and more details, we direct the reader to \autocite{garciaIntroductionModelSpaces2016} and \autocite{Nikolski}.
	
	Let $H^{\infty}(\oud)$ be the set of all holomorphic functions that are bounded on $\oud$, and let $H^2(\oud)$ be the Hardy-Hilbert space of $\oud$, which is the space of all holomorphic functions $f \in \mathcal{H}(\oud)$ such that \[\sup_{0 < r < 1}{\int_{\mathbb{T}}{|f(\zeta)|^2 \mathrm{d}m(\zeta)}} < \infty \] ($m$ is the Haar measure on $\mathbb{T}$) or, equivalently, such that \[
	\sum_{n=0}^{\infty}|a_n|^2 < \infty \quad \text{ if } f(z)=\sum_{n=0}^{\infty}a_nz^n.
	\]
	
	Let $S : H^2(\oud) \to H^2(\oud)$ be the \emph{unilateral shift}, defined by $S(f)(z) = zf(z)$. 
	For $f \in H^{\infty}(\oud)$, Fatou's theorem (see \textit{e.g.} \autocite[theorem 1.10]{garciaIntroductionModelSpaces2016}) states that $f$ has radial boundary values $f(\zeta)$, for almost every $\zeta \in \mathbb{T}$. A function $ u \in H^{\infty}(\oud)$ is said to be \emph{inner} if $|u(\zeta)|=1$ almost everywhere on $\mathbb{T}$. 
	
	If $u$ is an inner function, the corresponding \emph{model space} $\mathcal{K}_u$ is defined to be 
	\[
	\mathcal{K}_u := \left(uH^2(\oud)\right)^{\perp} = \left\{f \in H^2(\oud) \: : \: \langle f, uh \rangle =0, \: \forall \, h \in H^2(\oud)\right\}.
	\]
	We define the associated \emph{compressed shift} by $S_u : = P_u S \vert_{\mathcal{K}_u}$, where $P_u$ is the orthogonal projection from $H^2(\oud)$ onto $\mathcal{K}_u$. 
		
	 Now, let $\Theta$ be a finite Blaschke product with pairwise distinct zeros $\omega_1, \dots, \omega_n \in \oud$ and let $b_{\omega_k}(z)=\frac{z - \omega_k}{1 - \overline{\omega_k}z}$ denote a single Blaschke factor. Let $(\phi_1(z), \dots, \phi_n(z))$ denote the Takenaka--Malmquist--Walsh orthonormal basis (\autocite{garciaIntroductionModelSpaces2016, Nikolski}) of $\mathcal{K}_{\Theta}$, \textit{i.e.}
	\[
	\phi_1(z)=\frac{\sqrt{1- |\omega_1|^2}}{1-\overline{\omega_1}z} \quad \text{  and  } \quad \phi_k(z)=\left(\prod_{j=1}^{k-1}{b_{\omega_j}}\right)\frac{\sqrt{1- |\omega_k|^2}}{1-\overline{\omega_k}z} \quad k=2, \dots, n.
	\]
	Writing $S_{\Theta}$ with respect to the Takenaka--Malmquist basis gives the matrix representation $M_ {\Theta}$ with entries
	\[
	\left[M_{\Theta}\right]_{i,j} = \left\lbrace\begin{matrix}
		\omega_j & \text{if $i$=$j$} \\ 
		\prod_{k=i+1}^{j-1}{(-\overline{\omega_k})\sqrt{1-|\omega_i|^2}\sqrt{1-|\omega_j|^2}} & \text{ if $i$<$j$}  \\ 
		0  & \text{ if $i$>$j$}.
	\end{matrix}\right.
	\]
	
	It seems that the first appearance of this remarkable matrix was in \cite{Young2}; see also \cite{PtakLAA,PtakYoung,Nikolski,FoiasFrazho,Szehr}. In particular, for $n=2$ and $n=3$, we obtain the following matrices
	
	\begin{align*} & T_2 := \begin{pmatrix}
		\omega_1 & \sqrt{1-|\omega_1|^2}\sqrt{1-|\omega_2|^2} \\
		0 & \omega_2
	\end{pmatrix}, \\
 & T_3:= \begin{pmatrix}\omega_1 & \sqrt{1-|\omega_1|^2}\sqrt{(1-|\omega_2|^2} & -\overline{\omega_2}\sqrt{(1-|\omega_1|^2}\sqrt{(1-|\omega_3|^2)} \\ 0 & \omega_2 & \sqrt{1-|\omega_2|^2}\sqrt{(1-|\omega_3|^2}\\ 0 & 0 & \omega_3\end{pmatrix},
\end{align*}
which have been used to obtain the Schwarz-Pick and Beardon-Minda inequalities.

We now show how to obtain the equality case in the Beardon-Minda inequality using the matrix $T_3$. It follows from our proof of Theorem~\ref{beardon-minda} that $f$ satisfies \eqref{eq:BM} with equality if and only if $\|f(T_3)\| = 1=\|f\|_{\infty}$. The fact that $f$ is a finite Blaschke product of degree $\le 2$ has been proved by several authors, sometimes in relation to Crouzeix's conjecture. This can be generalized to $n$ points. We refer 
to the discussion in~\cite[Theorem 3.1]{bickelCrouzeixConjectureRelated2020}. An explicit description of the matrix which diagonalize $M_ {\Theta}$ is also given in~\cite{bickelCrouzeixConjectureRelated2020}.

We plan to return to the general Beardon-Minda type inequality $\|f(M_{\Theta})\| \leq 1$ in a future paper. 
	
	\subsection{A Beardon-Minda type lemma for derivatives}\label{s:beardon-minda-limit-case}
	
	We now investigate the case where $\omega_1=\omega_2=\omega_3=:\omega$. For a holomorphic function $f$ we use the 
 notation 
 $$\Gamma(z, f) = \frac{(1-|z|^2)|f'(z)|}{1-|f(z)|^2}.$$
 The Schwarz-Pick inequality for derivatives \eqref{eq:Schw-Pick-der} can then be expressed as $|\Gamma(z,f)| \le 1.$
 
We now give an operator theoretical proof of the following result, proved by Goluzin and \citeauthor{yamashitaPickVersionSchwarz1994} (see \autocite[Theorem 2]{yamashitaPickVersionSchwarz1994}).
	
	\begin{theo}\label{yamashita-theo}
		Let $f \in \mathcal{H}(\oud,\oud)$ and let $\Gamma(z, f) = \frac{(1-|z|^2)|f'(z)|}{1-|f(z)|^2}$.
		Then, for every $\omega \in \oud$, 
		\begin{equation}\label{yamashita}
		\left| \frac{\partial \Gamma(\omega, f)}{\partial \omega} \right| \leq \frac{1 - |\Gamma(\omega, f)|^2 }{1-|\omega|^2}.
		\end{equation}
		Moreover, equality holds if and only if $f$ is a Blaschke product of degree $\leq 2$.
	\end{theo}
	
	\begin{proof}
		Let $\omega \in \oud$, and let $T = \begin{pmatrix}
			\omega & \alpha & \beta \\
			0 & \omega & \alpha \\
			0 & 0 & \omega
		\end{pmatrix} \in \mathcal{M}_3(\mathbb{C})$, with $\alpha=1-|\omega|^2$ and $\beta=-\cj{\omega}(1-|\omega|^2)$.
		By Theorem~\ref{theo:gupta}, $T$ is a contraction. Moreover, we can easily check that for $f$ in the disk algebra we have 
  $$f(T) = \begin{pmatrix}
			f(\omega) & \alpha f'(\omega) & \frac{1}{2} \alpha^2 f''(\omega) + \beta f'(\omega) \\
			0 & f(\omega) & \alpha f'(\omega) \\
			0 & 0 & f(\omega)
		\end{pmatrix}.$$ 
  In this representation, the divided differences have been replaced in this limit case by first and second-order derivatives. By von Neumann's inequality, $f(T)$ is a contraction. Using Theorem~\ref{theo:gupta} we obtain:
		\[
		\left| \left(\frac{1}{2}\alpha^2 f''(\omega) + \beta f'(\omega) \right) \left( 1 - |f(\omega)|^2 \right) + \alpha^2 f'(\omega)^2\cj{f(\omega)} \right| \leq \left(1 - |f(\omega)|^2 \right)^2 - \left| \alpha f'(\omega) \right|^2,
		\]
		which is equivalent to \eqref{yamashita}. A proof of the equality case can be obtained using model spaces, as discussed in the preceding subsection.
	\end{proof}

The inequality \eqref{yamashita} can be rephrased in terms of \emph{Peschl's invariant derivatives}. Let $f \in \mathcal{H}(\oud, \oud)$, let $\omega \in \oud$, and consider the mapping
\begin{equation}\label{eq:peschl}
  g : z \in \oud \mapsto \frac{f\left(\frac{z+\omega}{1+\cj{\omega}z}\right)-f(\omega)}{1-\cj{f(\omega)}f\left( \frac{z+\omega}{1+\cj{\omega}z}\right)} \in \C.  
\end{equation}

Then $g$ is analytic on $\oud$ and $g(0)=0$. We have $g(z)=\sum_{n=1}^{\infty} \frac{D_nf(\omega)}{n!}z^n$, with $D_nf(\omega) := g^{(n)}(0)$. The quantities $D_nf(\omega)$ are called \emph{Peschl's invariant derivatives} (see \textit{e.g.} \autocite{kimInvariantDifferentialOperators2007}). 

The first two values of Peschl's invariant derivatives are explicitely computed as:

\begin{align*}
	& D_1f(\omega)= \frac{(1-|\omega|^2)f'(\omega)}{1-|f(\omega)|^2}, \\
	& D_2f(\omega) = \frac{(1-|\omega|^2)^2}{1-|f(\omega)|^2}\left[f''(\omega)- \frac{2\cj{\omega}f'(\omega)}{1-|\omega|^2}+\frac{2 \cj{f(\omega)}f'(\omega)^2}{1-|f(\omega)|^2}\right].
\end{align*}

With these notations, the Schwarz-Pick inequality for derivatives \eqref{eq:Schw-Pick-der} can be restated as 
 $|D_1f(\omega)| \leq 1$, while \eqref{yamashita} can be written as $|D_2f(\omega)| \leq 2( 1 -|D_1f(\omega)|^2)$.

 We refer to \autocite[Proposition 3.4]{choMultipointSchwarzPickLemma2012} for a different proof of \eqref{yamashita} and to Section \ref{section5} for a generalization to the polydisk.

\section{Operator versions of Beardon-Minda's inequality} \label{sect:3}
We move now to operator versions of the Schwarz-Pick and Beardon-Minda inequalities. The first operator generalization for the Schwarz-Pick inequality has been proved by Ky Fan in~\cite{FanAnalyticfunctionsproper1978}; the following discussion has been inspired by the recent paper~\autocite{jocicNoncommutativeSchwarzLemma2022}.

We recall the following theorem concerning the Sylvester equation
 $AX-XB = Y,$ which has been studied \textit{e.g.} in \autocite{bhatiaHowWhySolve1997, rosenblumOperatorEquationBXXA1956}.

\begin{theo}[Rosenblum,\autocite{bhatiaHowWhySolve1997}]\label{rosenblum-sylvester} Let $H, K$ be two Hilbert spaces. Let $A \in \mathcal{B}(H)$ and $B \in \mathcal{B}(K)$ be two operators with $\sigma(A) \cap \sigma(B) = \emptyset$. Then, for every $Y \in \mathcal{B}(H,K)$, the Sylvester equation $AX-XB=Y$ has a unique solution $X$. Moreover, if $\Gamma$ is a union of closed contours in the plane with total winding numbers $1$ around $\sigma(A)$ and $0$ around $\sigma(B)$, the solution can be expressed as 
\[ X = \frac{1}{2 i \pi} \int_{\Gamma} (A - \xi)^{-1} Y (B - \xi)^{-1} d\xi .\] 
\end{theo}

\subsection{An operator version of the Schwarz-Pick inequality}
The following result is a counterpart of \autocite[Theorem 3.5]{jocicNoncommutativeSchwarzLemma2022}. When specialized to scalars, it reduces to the Schwarz-Pick inequality for two distinct points.

\begin{theo}\label{schwarz-pick-op-sylvester}
	Let $H_1, H_2$ be two Hilbert spaces. Consider three contractions $W_1 \in \mathcal{B}(H_1)$, $W_2 \in \mathcal{B}(H_2)$ and $V \in \mathcal{B}(H_2, H_1)$. Assume that $\sigma(W_1) \cap \sigma(W_2) = \emptyset$, and that $f \in \mathcal{H}(\oud,\oud)$ is holomorphic on an open neighborhood of $\sigma(W_1) \cup \sigma(W_2)$. We denote by $X=X_{W_1, W_2,V}$ the unique solution of Sylvester's equation 
 \begin{equation} \label{eq:31} W_1X-XW_2= D_{W_ 1^*}VD_{W_2}.
 \end{equation}
	
	Then, there exists a contraction $Y \in \mathcal{B}(H_2,H_1)$ such that \[f(W_1)X-Xf(W_2)=D_{f(W_1)^*}YD_{f(W_2)}.\]
\end{theo}

\begin{proof}
	Let $T=\begin{bmatrix}
		W_1 & D_{W_1^*}VD_{W_2} \\
		0	& W_2
	\end{bmatrix}$. Denote $C = D_{W_1^*}VD_{W_2}$. By Parrott's theorem, $T$ is a contraction. Moreover, using \eqref{eq:31}, we have 
 \[ 
T = \begin{bmatrix}
	W_1 & C \\
	0 & W_2
\end{bmatrix} = \begin{bmatrix}
\id & -X \\
0 & \id
\end{bmatrix} \begin{bmatrix}
W_1 & 0 \\
0 & W_2
\end{bmatrix} \begin{bmatrix}
\id & X \\
0 & \id
\end{bmatrix}.
\] 
 Notice that $\sigma(T) \subset \sigma (W_1) \cup \sigma(W_2)$.  Indeed, for $\lambda \in \mathbb{C}$,  we have \[T - \lambda \id = \begin{bmatrix}\id & 0 \\ 0 & W_2- \lambda \id \end{bmatrix} \begin{bmatrix}
		\id & D_{W_1^*}VD_{W_2} \\ 0 & \id
	\end{bmatrix} \begin{bmatrix}
		W_1- \lambda \id & 0 \\ 0 & \id
	\end{bmatrix}.\]
	Therefore,  if $\lambda \not \in \sigma(W_1) \cup \sigma (W_2)$, then all factors in the previous decomposition are invertible and thus $\lambda \not \in \sigma (T)$. So it makes sense to speak about $f(T)$ and to write 
	\[
	f(T) = \begin{bmatrix}
		\id & -X \\
		0 & \id
	\end{bmatrix} \begin{bmatrix}
		f(W_1) & 0 \\
		0 & f(W_2)
	\end{bmatrix} \begin{bmatrix}
		\id & X \\
		0 & \id
	\end{bmatrix} = \begin{bmatrix}
	f(W_1) & f(W_1)X-Xf(W_2) \\
	0 & f(W_2)
	\end{bmatrix}.
	\]
	
	As $\|f\|_{\infty} \le 1$, we have $\|f(T)\| \leq 1$ by von Neumann's inequality. Thus,  by Parrott's theorem, there exists a contraction $Y\in \mathcal{B}(H_2,H_1)$ such that $f(W_1)X-Xf(W_2)=D_{f(W_1)^*}YD_{f(W_2)}$.
\end{proof}

\subsection{An operator version of the Beardon-Minda inequality}
Utilizing the analogue proof framework as employed in Theorem~\ref{schwarz-pick-op-sylvester}, we can deduce the following outcome for $3\times 3$ operator matrices.

\begin{theo}
		Let $H_1, H_2, H_3$ be three Hilbert spaces. Consider three contractions $W_1 \in \mathcal{B}(H_1)$, $W_2 \in \mathcal{B}(H_2)$ and $W_3 \in \mathcal{B}(H_3)$. Let $V_1 \in \mathcal{B}(H_2,H_1)$, $V_2 \in \mathcal{B}(H_3,H_2)$, and $V_3 \in \mathcal{B}(H_3,H_1)$ be contractions.
		 Assume that $\|W_2\| <1$ and that $\sigma(W_i) \cap \sigma(W_j) = \emptyset$, for all $1 \leq i < j \leq 3$. Suppose that $f \in \mathcal{H}(\oud,\oud)$ is holomorphic on an open neighborhood of $\sigma(W_1) \cup \sigma(W_2) \cup \sigma(W_3)$. 
		 Let $X_1, X_2, X_3$ be respectively the unique solution of Sylvester's equations \begin{align}
		     & W_1X_1-X_1W_2=D_{W_1^*}V_1D_{W_2}, \\
                & W_2X_2-X_2W_3=D_{W_2^*}V_2D_{W_3} \text{ and } \\
                & W_1X_3-X_3W_3=B-W_3X_1X_2+X_1W_2X_2,
		 \end{align} where 
   $$B=\left[D_{W_1^*}(\id-V_1V_1^*)D_{W_1^*} \right]^{1/2} V_3 \left[ D_{W_3}(\id - V_2^*V_2)D_{W_3} \right]^{1/2} - D_{W_1^*}V_1W_2^*V_2D_{W_3}.$$
		 Then, there exist three contractions $Y_1 \in \mathcal{B}(H_2,H_1), Y_2 \in \mathcal{B}(H_3,H_2), Y_3 \in \mathcal{B}(H_3,H_1)$ such that :

   \begin{numcases}{}
		 	f(W_1)X_1-X_1f(W_2)=D_{f(W_1)^*}Y_1D_{f(W_2)}, \\
		 	f(W_2)X_2-X_2f(W_3)=D_{f(W_2)^*}Y_2D_{f(W_3)}, \\
		 	f(W_1)X_3-X_3f(W_3)= 
		 		X_1f(W_2)X_2-X_1X_2f(W_3) \notag\\ + \left[D_{f(W_1)^*}(\id-Y_1Y_1^*)D_{f(W_1)^*} \right]^{1/2} Y_3 \left[ D_{f(W_3)}(\id - Y_2^*Y_2)D_{f(W_3)} \right]^{1/2} \notag\\ - D_{f(W_1)^*}Y_1f(W_2)^*Y_2D_{f(W_3)}.\label{eq:37}
		 \end{numcases}
\end{theo}
\begin{proof}
    Denote $A_1 = D_{W_1^*}V_1D_{W_2}$ and $A_2 = D_{W_2^*}V_2D_{W_3}$. Then, according to Theorem~\ref{theo:OperBM}, the operator
    \[
	T=\begin{bmatrix} W_1 & A_1 & B \\
		0 & W_2 & A_2 \\
		0 & 0 & W_3 \end{bmatrix}
  \]
  is a contraction. Notice also that  
    $X_1 \in \mathcal{B}(H_2,H_1)$, $X_2 \in \mathcal{B}(H_3,H_2)$ and $X_3 \in \mathcal{B}(H_3,H_1)$ are respectively the unique solutions of Sylvester's equations $W_1X_1-X_1W_2=A_1$, $W_2X_2-X_2W_3=A_2$ and $W_1X_3-X_3W_3=B-W_3X_1X_2+X_1W_2X_2$. 

    In analogy with some computations in the Heisenberg group, we can write
    \[
	T = \begin{bmatrix} W_1 & A_1 & B \\
		0 & W_2 & A_2 \\
		0 & 0 & W_3 \end{bmatrix} = \begin{bmatrix} \id & -X_1 & X_1X_2-X_3 \\
		0 & \id & -X_2 \\
		0 & 0 & \id \end{bmatrix} \begin{bmatrix}
		W_1 & 0 & 0 \\
		0 & W_2 & 0 \\
		0 & 0 & W_3
		\end{bmatrix} \begin{bmatrix} \id  & X_1 & X_3 \\
		0 & \id & X_2 \\
		0 & 0& \id\end{bmatrix}.
	\]
 This diagonalization allows one to write the $3\times 3$ operator matrix of $f(T)$, which is a contraction by von Neumann's inequality:
   \[ f(T) =
 \begin{bmatrix} \id & -X_1 & X_1X_2-X_3 \\
		0 & \id & -X_2 \\
		0 & 0 & \id \end{bmatrix} \begin{bmatrix}
		f(W_1) & 0 & 0 \\
		0 & f(W_2) & 0 \\
		0 & 0 & f(W_3)
		\end{bmatrix} \begin{bmatrix} \id  & X_1 & X_3 \\
		0 & \id & X_2 \\
		0 & 0& \id\end{bmatrix}.  
 \]
Thus the matrix of $f(T)$ is given by
\[
	\begin{bmatrix} f(W_1) & f(W_1)X_1-X_1f(W_2) & f(W_1)X_3-X_1f(W_2)X_2+(X_1X_2-X_3)f(W_3) \\
		0 & f(W_2) & f(W_2)X_2-X_2f(W_3) \\
		0 & 0 & f(W_3) \end{bmatrix}.
  \]
 We apply again Theorem~\ref{theo:OperBM}. 
\end{proof}

In the scalar case, the condition \eqref{eq:37} is equivalent to the Beardon-Minda inequality.

\section{Schwarz-Pick inequalities for the polydisk}

Let $n \in \mathbb{N}^*$. For $\mpt{\omega}=(\omega_1, \cdots, \omega_n) \in \mathbb{C}^n$, we denote $\|\mpt{\omega}\| = \sup_{1 \leq i \leq n} |\omega_i|$ the sup norm.

\subsection{Using von Neumann inequality for tuples of two by two matrices}
It is a fascinating observation in operator theory that an analogue of the von Neumann inequality holds for the bidisk (Ando's theorem), but does not extend to the polydisk $\oud^n$ for $n\ge 3$. However, as proved by Drury \cite{druryRemarksNeumannInequality1983} and Knese \cite{kneseNeumannInequalityTimes2016}, there is an analogue for tuples of $2\times 2$ and $3\times 3$ commuting matrices.

\begin{lem}[Drury and Knese; see \autocite{druryRemarksNeumannInequality1983}, \autocite{kneseNeumannInequalityTimes2016}]\label{fu-russo}
	Let $T_1,...,T_n$ be mutually commuting $2\times 2$ or $3\times 3$ contractions, and let $p \in \mathbb{C}[X_1, \cdots, X_n]$. Then, we have 
	\[
	\|p(T_1,...,T_n)\| \leq \|p\|_{\infty}:= \sup\{|p(z_1,\cdots z_n)| : \mpt{z} \in \oud^n\}.
	\]
\end{lem}

This leads to operator theoretical proofs of the following known (\autocite[Lemma 7.5.6]{rudinFunctionTheoryPolydiscs1969}) Schwarz-Pick inequalities for the polydisk. 

\begin{theo}\label{schwarz-pick-polydisk} (a) Let $f \in \mathcal{H}(\oud^n, \oud)$ and let $\mpt{a}=(a_1, \dots , a_n)$, $ \mpt{b}=(b_1, \dots, b_n) \in \oud^n$. Then
\begin{equation} \label{eq:SP-pol}
    \left| \frac{f(a_1, \cdots, a_n) - f(b_1, \cdots, b_n)}{1 - \overline{f(a_1, \cdots, a_n)} f(b_1, \cdots , b_n)} \right| \leq \max_{1 \leq i \leq n} \left| \frac{a_i - b_i}{1 - \overline{a_i}b_i} \right|.
\end{equation}

 (b) Let $f \in \mathcal{H}(\oud^n,\oud)$ and let $\mpt{a}=(a_2, \dots, a_n) \in \oud^n$.
	Then,
 \begin{equation} \label{eq:SP-pol-der}
	\sum_{i=1}^{n}(1-|a_i|^2)\left|\frac{\partial f(\mpt{a})}{\partial z_i}\right| \leq 1- |f(\mpt{a})|^2.
	\end{equation}
\end{theo}

\begin{proof}
	(a) \quad We first observe that the result is obvious whenever $\mpt{a} = \mpt{b}$ or $f(\mpt{a}) = f(\mpt{b})$. Therefore, in the following, we assume $\mpt{a} \neq \mpt{b}$ and $f(\mpt{a}) \neq f(\mpt{b})$.
	
	For $1 \leq i \leq n$, let $$T_i = \begin{pmatrix}
		a_i & d(a_i-b_i)\\ 
		0 & b_i \end{pmatrix},$$ 
  with 
  $$d = \min_{1 \leq i \leq n} \sqrt{\frac{(1-|a_i|^{2})(1-|b_i|^2)}{|a_i-b_i|^2}}.$$ 
Here, whenever $a_i = b_i$, we make the convention that $\sqrt{\frac{(1-|a_i|^{2})(1-|b_i|^2)}{|a_i-b_i|^2}} = + \infty$. As we assume that $\mpt{a}\neq \mpt{b}$, this cannot happen for all the indices $i$.

It can be easily verified that the matrices $T_i$ are mutually commuting and that $\|T_i\| \le 1$. By induction it can be shown that for all $i \in \llbracket 1, n \rrbracket$, for all $k_i \in \mathbb{N}$, 	
$$T_i^{k_i}=\begin{pmatrix}
		a_i^{k_i} & d \left(a_i^{k_i}-b_i^{k_i} \right) \\ 
		0 & b_i^{k_i}
	\end{pmatrix}.$$ 
Let $p \in \mathbb{C}[X_1, \dots, X_n]$ be a polynomial with $\|p\|_{\infty} < 1$. We have 
$$p(T_1, \cdots, T_n) = \begin{pmatrix}
		p(\mpt{a}) & d(p(\mpt{a}) - p(\mpt{b})) \\ 
		0 & p(\mpt{b})
	\end{pmatrix}.$$ 
Drury's result implies that $\|p(T_1, \dots, T_n)\| \leq 1$. As in the one variable case, a computation gives the Schwarz-Pick inequality \eqref{eq:SP-pol} for $p$. By an approximation argument, \eqref{eq:SP-pol} holds also for functions in the polydisk algebra. 
Now, if $f \in \mathcal{H}(\oud^n,\oud)$, consider the family of functions $(f_r)_{0 < r < 1}$ defined by $f_r(z_1, \cdots, z_n) = f(rz_1, \dots, rz_n)$. For all $r \in \, ]0, 1[$, $f_r$ is in the polydisk algebra and, thus, $f_r$ satisfies \eqref{eq:SP-pol}. Then, let $r \to 1^-$ to conclude the proof.

(b) \quad The proof follows the same method as that of Theorem~\ref{schwarz-pick-polydisk}. Let $\mpt{a}=(a_2, \dots, a_n) \in \oud^n$ and let $p \in \mathbb{C}[X_1, \dots, X_n]$ be a polynomial with $\|p\|_{\infty} < 1$. For $1 \leq k \leq n$, let $T_k = \begin{pmatrix}
		a_k & \gamma_k \\ 
		0 & a_k \end{pmatrix}$, where $\gamma_k = e^{i \theta_k} (1 - |a_k|^2)$, for some $\theta_k \in [0, 2 \pi[$ to be chosen later on. For all $k \in \llbracket 1 , n \rrbracket$, $\|T_k\| \leq 1$, and, for all $k,l \in \llbracket 1 , n \rrbracket$, $T_kT_l=T_lT_k$.
	We have
	\[
	p(T_1, \cdots, T_n) = \begin{pmatrix}
		p(\mpt{a}) & \sum_{k=1}^{n} \gamma_k \frac{\partial p(\mpt{a})}{\partial z_k}  \\ 
		0 & p(\mpt{a})
	\end{pmatrix}.
	\]
	Again, by Lemma~\ref{fu-russo}, we get $\|p(T_1, \dots, T_n)\| \leq 1$. Therefore
	\begin{equation}\label{schwarz-pick-polydisk-limite-etape-1}
		\left|\sum_{k=1}^{n} \gamma_k \frac{\partial p(\mpt{a})}{\partial z_k} \right| \leq 1 - |p(\mpt{a})|^2.
	\end{equation}
	Now, let $t_k =\frac{\partial p(a_1, a_2)}{\partial z_k}$, $0 \leq k \leq n$. We write $t_k = |t_k|e^{i \text{Arg}(t_k)}$ and we set  $\theta_k = -\text{Arg}(t_k)$. \\
	With this choice we obtain $\gamma_k t_k = (1 - |a_k|^2) \, |t_k|$. Replacing in \eqref{schwarz-pick-polydisk-limite-etape-1} we get
	\[
	\sum_{i=1}^{n}(1-|a_i|^2)\left|\frac{\partial p(\mpt{a})}{\partial z_i}\right| \leq 1- |p(\mpt{a})|^2 .
	\]
	We conclude by using an approximation argument.
\end{proof}

\begin{rema}
    The study of the case of equality in the Schwarz-Pick inequalities for the polydisk is an interesting problem. Knese~\cite{kneseSchwarzLemmaPolydisk2007a} studied the equality case in \eqref{eq:SP-pol-der} using operator-theoretical methods (transfer functions) and described which functions play the role of automorphisms of the disk in this context–they turn out to be rational inner functions in the Schur-Agler class of the polydisk with an \emph{added} symmetry constraint. 
\end{rema}

\subsection{Peschl's invariant derivatives in several variables} \label{section5}

The inequalities from Section \ref{s:beardon-minda-limit-case} can be extended to analytic functions of several variables.

Let $n \in \mathbb{N}^*$, let $f \in \mathcal{H}(\oud^n,\oud)$, and fix a vector $\mpt{\omega}=(\omega_1, \cdots, \omega_n)$ in $\oud^n$.
Similarly as in the one variable case, we define
\[
g : \mpt{z}=(z_1, \cdots , z_n) \in \oud^n \mapsto \frac{f\left(\frac{z_1+\omega_1}{1+\cj{\omega_1}z_1}, \cdots , \frac{z_n+\omega_n}{1+\cj{\omega_n}z_n }\right)-f(\omega_1, \cdots, \omega_n)}{1-\cj{f(\omega_1, \cdots, \omega_n)}f\left(\frac{z_1+\omega_1}{1+\cj{\omega_1}z_1}, \cdots , \frac{z_n+\omega_n}{1+\cj{\omega_n}z_n }\right)} \in \C
\]
and then write 
\[
g(z_1, \cdots, z_n) = \sum_{j_1, \cdots , j_n=0}^{\infty} \frac{\partial^{j_1+\cdots+j_n}g(0, \cdots, 0)}{\partial^{j_1} z_1 \cdots \partial^{j_n}z_n}z_1^{j_1}\cdots z_n^{j_n} = \sum_{j_1, \cdots , j_n=0}^{\infty} a_{j_1, \cdots, j_n}z_1^{j_1}\cdots z_n^{j_n}.
\] 

For $k \in \llbracket 1, n \rrbracket$, let $D_kf(\mpt{w})= \partial^kg(0, \dots, 0) = \sum_{j_1+\dots+j_n=k}a_{j_1, \cdots, j_n}$.
A straightforward computation gives :
\begin{align*}
	 D_1f(\mpt{\omega}) & = \sum_{j=1}^{n} \frac{1-|\omega_j|^2}{1-|f(\mpt{\omega})|^2} \cdot \frac{\partial f}{\partial z_j}(\mpt{\omega}), \\
	D_2f(\mpt{\omega}) & = \sum_{j=1}^{n} \frac{\partial^2g(0, \cdots, 0)}{\partial^2 z_j} + 2 \sum_{1 \leq j < k \leq n} \frac{\partial^2 f(0, \cdots, 0)}{\partial z_j \partial z_k} \\
					  & =  \begin{multlined}[t]
					  	 \sum_{j=1}^{n} \frac{(1-|\omega_j|^2)^2}{1-|f(\mpt{\omega})|^2}\left( \frac{\partial^2 f (\mpt{w})}{\partial^2 z_j}  + \frac{2 \cj{f(\mpt{w})}}{1-|f(\mpt{w})|^2} - \frac{2 \cj{\omega_j}}{1-|\omega_j|^2} \cdot \frac{\partial f(\mpt{\omega})}{\partial z_j} \right) \ \\
					  	+ 2 \sum_{1 \leq j < k \leq n} \frac{(1-|z_j|^2)(1-|z_k|^2)}{1-|f(\mpt{\omega})|^2} \left( \frac{\partial f(\mpt{\omega})}{\partial z_j \partial z_k} + \frac{2 \cj{f(\mpt{\omega})}}{1-|f(\mpt{\omega})|^2} \cdot \frac{\partial f (\mpt{\omega})}{\partial z_j} \cdot \frac{\partial f (\mpt{\omega})}{\partial z_k}\right).
					  	\end{multlined}
\end{align*}

With the same method of proof as before, we can arrive at the following result.
\begin{theo}\label{ineq-peschl-mult-mult-varia}
	For $n \in \mathbb{N}^*$ let $\mpt{w} = (\omega_1, \dots, \omega_n) \in \oud^n$ and consider $f \in \mathcal{H}(\oud^n, \oud)$. Then, we have:
	\begin{equation}\label{ineq-peschl-mult-mult-varia-eq}
	|D_2f(\mpt{\omega})| \leq 2 (1 - |D_1f(\mpt{\omega})|^2).
	\end{equation}
\end{theo}

\begin{proof}
	For $1 \leq k \leq n$, let 
 $$T_k=\begin{pmatrix}
		\omega_k & \alpha_k & \beta_k \\
		0 & \omega_k & \alpha_k \\
		0 & 0 & \omega_k
	\end{pmatrix} \in \mathcal{M}_3(\mathbb{C}),$$
 with $\alpha_k = 1-|\omega_k|^2$ and $\beta_k=-\cj{\omega_k}(1-|\omega_k|^2)$. By Theorem~\ref{theo:gupta}, $T_k$ is a contraction, for all $k \in \llbracket 1 , n \rrbracket$. Moreover, for all $1 \leq k,j \leq n$, $T_jT_k=T_kT_j$. Therefore, by Knese's result, $p(T_1, \dots, T_n)$ is a contraction, for every $p \in \mathbb{C}[X_1, \dots, X_n]$ with $\|p\|_{\infty} < 1$. Moreover, it is easy to check that \[p(T_1, \dots, T_n) = \begin{pmatrix}
		p(\mpt{\omega}) & \gamma_1 & \gamma_2 \\
		0 & p(\mpt{\omega}) & \gamma_1 \\
		0 & 0 & p(\mpt{\omega})
	\end{pmatrix},\] with 
	\begin{align*}
	& \gamma_1 = \sum_{j=1}^{n} \alpha_j \frac{\partial p(\mpt{\omega})}{\partial z_j}, \\ 
 & \gamma_2 = \frac{1}{2} \sum_{j=1}^{n} \alpha_j^2 \frac{\partial^2 p(\mpt{w})}{\partial^2 z_j} + \sum_{1 \leq j < k \leq n} \alpha_j \alpha_k \frac{\partial^2 p(\mpt{w})}{\partial z_j \partial z_k} + \sum_{j=1}^{n} \beta_j \frac{\partial p (\mpt{w})}{\partial z_j}.
	\end{align*}
	By Theorem~\ref{theo:gupta}, we obtain :
	\begin{equation}
	\left| \gamma_2 \left(1-|p(\mpt{\omega})|^2\right) + \gamma_1^2 \cj{p(\mpt{\omega})} \right| \leq (1-|p(\mpt{\omega})|^2)^2 - |\gamma_1|^2,
	\end{equation}
	which is equivalent to \eqref{ineq-peschl-mult-mult-varia-eq} for polynomials. The inequality extends to all functions $f \in \mathcal{H}(\oud^n, \oud)$.
\end{proof}

\subsection{Distinguished varieties and Schwarz-Pick inequalities}
In the bidisk case, the refined version of Ando's inequality by Agler and McCarthy~\cite{AM:acta} results in corresponding enhancements of Schwarz-Pick type inequalities.

We start by recalling the notion of distinguished variety introduced in~\cite{AM:acta}. 
A \emph{distinguished variety} is a set of the form $V\cap \cud^2$, where $V$ is an algebraic set in $\C^2$ (so there is a polynomial $q\in \C[z,w]$ such that $V = \{(z,w)\in\oud^2 : q(z,w) = 0\}$) with the property that 
$$\overline{V} \cap \partial(\oud^2) = \overline{V} \cap \T^2.$$
Therefore a distinguished variety is the trace on $\oud^2$ of a one-dimensional complex algebraic variety $V$ in $\C^2$ such that $V$ intersects $\oud^2$ and exits the bidisk through its distinguished boundary, $\T^2$, without
intersecting any other part of its topological boundary. A distinguished variety has (\cite{AM:acta}) the following determinantal representation
\begin{equation}\label{eqn:determ}
V\cap \oud^2=\left\{ (z,w)\in \oud^2\,:\, \text{ det } (\Psi(z)-w\id)=0 \right\}
\end{equation}
for some matrix-valued rational function $\Psi$ on the unit disc that is unitary on the unit circle.

Agler and McCarthy proved in~\cite{AM:acta} that for any pair of commuting contractive matrices $(T_1,T_2)$ without unimodular eigenvalues, there is a distinguished variety $V\cap \oud^2$ such that the von-Neumann
inequality holds on $V\cap \oud^2$ for any polynomial $p$ in $\C[z_1,z_2]$, \, \emph{i.e.}
\begin{equation}\label{eqn:Ando}
\| p(T_1,T_2) \|\leq \sup_{(z_1,z_2)\in V \cap \oud^2} |p(z_1,z_2)|.
\end{equation}

\begin{theo}
(a) Let $(a_1,a_2)$ and $(b_1,b_2)$ be two points in the bidisk $\oud^2$. Then there is a distinguished variety $V\cap\oud^2$ such that 
the Schwarz-Pick inequality 
\begin{equation} \label{eq:bidisk}
    \left| \frac{f(a_1,a_2) - f(b_1,b_2)}{1 - \overline{f(a_1,a_2)} f(b_1,b_2)} \right| \leq \max\left\{ \left| \frac{a_1 - b_1}{1 - \overline{a_1}b_1}\right|, \left| \frac{a_2 - b_2}{1 - \overline{a_2}b_2} \right|\right\}
\end{equation}
holds for
any function $f$ which is  holomorphic on the bidisk $\oud^2$ and continuous on $\cud^2$
with $$\sup_{(z_1,z_2)\in V \cap \oud^2} |f(z_1,z_2)| \le 1.$$

(b) Let $(a_1,a_2)$ and $(b_1,b_2)$ be two points in the bidisk $\oud^2$. Then there is a distinguished variety $V\cap\oud^2$ such that 
the Schwarz-Pick inequality  \eqref{eq:bidisk}
holds for
any function $f$ which is  holomorphic in the bidisk $\oud^2$ and for which there is a sequence of positive real number $(r_n)$ convergent to $1$ with $r_n < 1$ such that
$$\sup_{n\ge 1,(z_1,z_2)\in V \cap \oud^2} |f(r_nz_1,r_nz_2)| \le 1.$$
\end{theo}

\begin{proof}
Consider the matrices
$$T_1 = \begin{pmatrix}
		a_1 & d(a_1-b_1) \\ 
		0 & b_1 \end{pmatrix}, \quad T_2 = \begin{pmatrix}
		a_2 & d(a_2-b_2) \\ 
		0 & b_2 \end{pmatrix},$$ 
  with 
  $$d = \min \left\{\sqrt{\frac{(1-|a_1|^{2})(1-|b_1|^2)}{|a_1-b_1|^2}}, \sqrt{\frac{(1-|a_2|^{2})(1-|b_2|^2)}{|a_2-b_2|^2}} \right\},$$ 
with the same conventions as in the proof of Theorem~\ref{schwarz-pick-polydisk}, (a).  Following~\cite{AM:acta}, we can also assume that $T_1$ and $T_2$ are jointly diagonalizable (this is the first case in the proof of~\cite[Theorem 3.1]{AM:acta}). It follows from the result proved in~\cite{AM:acta} (see also~\cite[p. 211]{aglerOperatorAnalysisHilbert2020} for details and unexplained terminology) that there is a distinguished variety $V$ such that $T=(T_1,T_2)$ can be extended to a pair of commuting unitaries $U=(U_1,U_2)$ with spectrum $\sigma(U) = \overline{V} \cap \partial(\oud^2) = \overline{V} \cap \T^2$. As $f$ is in the bidisk algebra, $f(T)$ and $f(U)$ are well-defined and $f(T)$ is a restriction of $f(U)$ to $\C^2\times \C^2$. We obtain, as in~\cite{AM:acta}, that 
\begin{equation}\label{eqn:Andof}
\| f(T_1,T_2) \|\leq \sup_{(z_1,z_2)\in V \cap \oud^2} |f(z_1,f_2)|.
\end{equation}
Therefore $f(T_1,T_2)$ is a contraction and the proof of Theorem~\ref{schwarz-pick-polydisk}, (a), implies that inequality \eqref{eq:bidisk} holds true. The second part, (b), follows from (a) applied to the functions $f(r_nz_1,r_nz_2)$ and then making $n\to\infty$. 
\end{proof}
The following result follows in a similar manner from the Agler and McCarthy result and the proof of Theorem~\ref{ineq-peschl-mult-mult-varia}.
\begin{theo}\label{ineq-peschl-disting}
	Let $\mpt{w} = (\omega_1, \omega_2) \in \oud^2$. Then there exists a distinguished variety $V\cap\oud^2$ such that 
	\begin{equation}\label{ineq-peschl-mult-disting}
	|D_2f(\mpt{\omega})| \leq 2 (1 - |D_1f(\mpt{\omega})|^2)
	\end{equation}
 for every $f\in \mathcal{A}(\cud^2)$ with $$\sup_{(z_1,z_2)\in V \cap \oud^2} |f(z_1,z_2)| \le 1.$$
\end{theo}
Some Nevanlinna–Pick interpolation problems on distinguished varieties in the bidisk have been studied in \cite{JKM}.
\section{Higher order Schwarz-Pick inequalities} \label{sect:5}
Let $f \in \mathcal{H}(\oud, \oud)$ be an analytic function of $\oud$ into itself with $f(z) = \sum_{n=0}^{\infty} a_nz^n$. It has been proved by F.W.~Wiener that for each $k\ge 1$ we have 
\begin{equation}\label{eq:wiener}
    |a_k| \le 1-|a_0|^2.
\end{equation}
We refer for instance to \cite{paulsenBohrInequality2002} for an operator theoretical proof of this inequality and for applications to Bohr's phenomenon. For $k=1$, the inequality \eqref{eq:wiener} gives 
$|f'(0)| \le 1-|f(0)|^2$. Applying this inequality to $F(z)=f((\omega+z)/1+\overline{\omega}z)$, for a fixed $\omega \in \oud$, we obtain the Schwarz-Pick inequality \eqref{eq:Schw-Pick-der}. For an arbitrary $k$, a similar reasoning has been used by Ruscheveyeh~\cite{ruscheweyhSerdica} to obtain the following sharp higher-order inequality for an analytic function $f \in \mathcal{H}(\oud,\oud)$, $z\in\oud$ and $k\ge 1$:
\begin{equation}\label{eq:rus}
    |f^{(k)}(z)| \le \frac{k!(1-|f(z)|^2)}{(1-|z|)^k(1+|z|)}.
\end{equation}

We prove in this section some results related to the estimate \eqref{eq:wiener}.

\begin{theo}\label{th:higher}
    Let $f : \oud \mapsto \cud$ be an analytic function. Assume that $$f(z) = \sum_{n=0}^{\infty} a_nz^n\quad (z\in\oud).$$ Then, for each $n\ge 1$ and each $k\ge 1$ we have 
    \begin{equation}\label{eq:53}
        \left|a_{n+k}(1-|a_0^2|)+a_na_k\overline{a}_0\right|^2 \le
        \left[(1-|a_0|^2)^2-|a_n|^2\right]\cdot \left[(1-|a_0|^2)^2-|a_k|^2\right].
    \end{equation}
\end{theo}

\begin{proof}
    As $\|f\|_{\infty} \le 1$, the multiplication operator $M_f$ given by $M_f(g) = fg$ acts contractively on the Hardy space $H^2(\oud)$. Recall that $\{z^n : n\ge 0\}$ is an orthonormal basis of $H^2(\oud)$. The compression $T=P_KM_f\mid K$ of $M_f$ to the $3$-dimensional Euclidean space $K = \textrm{span}(1,z^n,z^{n+k})$ is also a contraction. The matrix of $T$ is given by 
    $$T=\begin{pmatrix}
		a_0 & a_n & a_{n+k} \\
		0 & a_0 & a_k \\
		0 & 0 & a_0
	\end{pmatrix}.$$
 Then \eqref{eq:53} is a consequence of Theorem~\ref{theo:gupta}. 
\end{proof}

When $a_0 = 0$ we obtain the following consequence.
\begin{cor}
    Let $f$ be a analytic function of $\oud$ into $\cud$ with $f(0)=0$ and $f(z) = \sum_{n=1}^{\infty} a_nz^n$ for $z\in\oud$. Then 
    \begin{equation}\label{eq:54}
        \left|a_{n+k}\right| \le
        \sqrt{1-|a_n|^2}\cdot \sqrt{1-|a_k|^2}.
    \end{equation}
\end{cor}

For $n=k=1$, and $a_0=0$, we obtain the inequality $|a_2|\le 1-|a_1|^2$. Applying this inequality to \eqref{eq:peschl} we obtain Yamashita's inequality $|D_2f(\omega)| \le 2(1 - |D_1f(\omega)|^2)$.

The following consequence is an improvement of Wiener's inequality~\eqref{eq:wiener}.

\begin{cor}
    Let $f$ be a analytic function of $\oud$ into $\cud$ with $f(z) = \sum_{n=0}^{\infty} a_nz^n$ for $z\in\oud$. Then 
    $$ 1 - |a_0|^2 - |a_n| \ge \frac{\left|a_{2n}(1-|a_0^2|)+a_n^2\overline{a}_0\right|}{2(1-|a_0|^2)}.$$
\end{cor}

\begin{proof}
    Applying \eqref{eq:53} for $k=n$ we obtain 
$$ \left|a_{2n}(1-|a_0^2|)+a_n^2\overline{a}_0\right| \le
        \left[(1-|a_0|^2)^2-|a_n|^2\right].$$
Therefore 
$$ 1 - |a_0|^2 - |a_n| \ge \frac{\left|a_{2n}(1-|a_0^2|)+a_n^2\overline{a}_0\right|}{1-|a_0|^2 + |a_n|} \ge \frac{\left|a_{2n}(1-|a_0^2|)+a_n^2\overline{a}_0\right|}{2(1-|a_0|^2)}.$$
The proof is complete.
\end{proof}

\section{Appendix} \label{sect:6}

The objective of this Appendix is to revisit Parrott's theorem as stated in Theorem~\ref{parrott}. We adopt the approach presented by \citeauthor{davisNormpreservingDilationsTheir1982} in \autocite{davisNormpreservingDilationsTheir1982}, making some modifications, particularly regarding the selection of solutions with minimal norms. 

This appendix is primarily intended for readers interested in Schwarz-Pick inequalities who may not have an extensive background in operator theory.

We start by recalling the following lemma.

\begin{lem}[Douglas \autocite{douglasMajorizationFactorizationRange1966}]\label{douglas}
	Let $L, M_1, M_2$ be Hilbert spaces. Suppose that $A \in \mathcal{B}(L,M_1)$, $B \in \mathcal{B}(L,M_2)$ and $c \geq 0$.
	Then, $B^*B \leq c^2 A^* A$ if and only if  there exists $C \in \mathcal{B}(M_1,M_2)$ such that \begin{equation}\label{eq-douglas}
		\begin{cases} B=CA, \\ \|C\| \leq c \end{cases}
	\end{equation}
	
	Moreover, if it is the case, there exists a unique operator $C_0$ satisfying \eqref{eq-douglas} such that $\text{Im}(A)^{\perp} \subset \text{Ker}(C_0)$. The operator $C_0$ satisfies \[\|C_0\|^2 = \inf \{ \, \|C\|^2 \, : \, C \text{ satisfies } \eqref{eq-douglas}\} = \inf \{ \mu \geq 0 \, : \, B^*B \leq  \mu A^*A\}\] and will thus be referred as the \emph{minimal solution} of the equation $B=CA$.
\end{lem}

From this lemma, we deduce the following result about column matrices. Recall that the defect operator of $B$ is given by $D_B = (\id-B^*B)^{1/2}$. 

\begin{prop}\label{parrott-vect-col}
	Let $H, K_1, K_2$ be Hilbert spaces. Suppose that $A \in \mathcal{B}(H,K_1)$ and $B \in \mathcal{B}(H,K_2)$ are contractions.
	Then, $\begin{bmatrix}
		A \\B
	\end{bmatrix} : H_1 \to K_1 \oplus K_2$ is a contraction if and only if there exists a contraction $V \in \mathcal{B}(H, K_1) $ such that $A = V D_B$.
	
	Moreover, if it is the case, there exists a unique contraction $V_0$ such that $A = V_0 D_B$ and $\operatorname{Im}(D_B)^{\perp} \subset \operatorname{Ker}(V_0)$. Then $V_0$ satisfies \[\|V_0\| = \inf\{ \, \|V\| \, : \,  A = V D_B \, \}\] and will thus be referred as the \emph{minimal solution of the equation $A = V D_B$}.
\end{prop}

\begin{proof} The column matrix $\begin{bmatrix}
			A \\ B
		\end{bmatrix}$ is a contraction if and only if $A^*A \leq \id - B^* B = D_B^* D_B$. Using Lemma~\ref{douglas}, we obtain $A = V D_B$ with $\|V\| \le 1$.
\end{proof}

\begin{cor}\label{parrott-col-iso}
	Let $H, K_1, K_2, A, B$ be as in \Cref{parrott-vect-col}, and let $U \in \mathcal{B}(H)$ be an arbitrary (but fixed) isometry.
	Then, $\begin{bmatrix}
		A \\ B
	\end{bmatrix} : H_1 \to K_1 \oplus K_2$ is a contraction if and only if there exists a contraction $V \in \mathcal{B}(H, K_1) $ such that $A = V U D_B$.
	
	Moreover, if it is the case, there exists a unique contraction $V_0$ such that $A=V_0UD_B$ and $\operatorname{Im}(U D_B)^{\perp}  \subset \operatorname{Ker}(V_0)$. Then the operator $V_0$ satisfies \[\|V_0\| = \inf \{ \, \|V\| \, : \, A=VUD_B \,\} \] and will thus be referred as the \emph{minimal solution} of the equation $A=VUD_B$.
\end{cor}

\begin{proof} It is enough to prove the sufficiency part.
By \Cref{parrott-vect-col}, if $\begin{bmatrix}
		A \\ B 
	\end{bmatrix}$ is a contraction, there exists a contraction $W \in \mathcal{B}(H, K_1)$ such that $A = W D_B$. Moreover, $W$ can be chosen such that $W=0$ on $\text{Im}(D_B)^{\perp}$ (and in this case, the minimal solution $W_0$ is unique).
 
    Now, let $V=WU^*$. As $U$ is an isometry, it is easy to see that $V$ is a contraction and that $VUD_B=WD_B=A$. 
    Moreover, $V=0$ on $\text{Im}(UD_B)^{\perp}$. Indeed, let $x \in \text{Im}(UD_B)^{\perp}$. For all $x' \in H$, $\langle x, UD_Bx' \rangle = 0$, which can be rewritten $\langle U^*x, D_Bx' \rangle = 0$. Thus, for $x \in \text{Im}(UD_B)^{\perp}$, $U^*x \in \text{Im}(D_B)^{\perp}$ and, then, $Vx=WU^*x = 0$ (by minimality of $W$). It is moreover easy to see that there exists a unique $V$ such that $A=VUD_B$ and $V=0$ on $\text{Im}(D_B)^{\perp}$.
    \end{proof}

\begin{cor}\label{parrott-vect-ligne}
	Let $H_1, H_2, K$ be Hilbert spaces. Suppose that $A \in \mathcal{B}(H_1,K)$ and $B \in \mathcal{B}(H_2,K)$ are contractions.
	Then, $\begin{bmatrix}
		A & B
	\end{bmatrix} : H_1 \oplus H_2 \to K$ is a contraction if and only if there exists a contraction $V \in \mathcal{B}(H_1,K)$ such that $A = D_{B^*}V$.
	
	Moreover, if it is the case, there exists a unique contraction $V_0$ such that $A = D_{B^*} V_0$ and $\operatorname{Im}(D_{B^*})^{\perp} \subset \operatorname{Ker}(V_0^*)$. We have \[\|V_0\| = \inf\{ \, \|V\| \, : \,  A = D_{B^*}V \, \}.\] The operator $V_0$ will be referred as the \emph{minimal solution} of the equation $A = D_{B^*}V$.
\end{cor}
\begin{proof}
	Observe that $\begin{bmatrix}
		A & B
	\end{bmatrix}$ is a contraction if and only if $\begin{bmatrix}
		A & B
	\end{bmatrix}^* 
	= \begin{bmatrix}
		A^* \\ B^*
	\end{bmatrix}$
	 is a contraction, and then apply \Cref{parrott-vect-col}.
\end{proof}
\begin{proof}[Proof of Theorem \ref{parrott}]
	First of all, the existence of two contractions $Z \in \mathcal{B}(H_1, K_1)$ and $Y \in \mathcal{B}(H_2,K_2)$ such that $D=D_{C^*}Y$ and $A=ZD_C$ comes from \Cref{parrott-vect-col} and \Cref{parrott-vect-ligne}, as $\begin{bmatrix}
		A \\ C
	\end{bmatrix}$ and $\begin{bmatrix}
		C & D
	\end{bmatrix}$ are contractions. We denote the minimal solutions by $Y_0$, and respectively $Z_0$.
	
	Set $\mathbf{A}=\begin{bmatrix}A & B \end{bmatrix}$ and $\mathbf{B}=\begin{bmatrix}C & D \end{bmatrix}$, so that we have $T= \begin{bmatrix}
		\mathbf{A} \\ \mathbf{B}
	\end{bmatrix}$, with $||\mathbf{B}|| \leq 1$.
	Now, using that $T D_T = D_{T^*}T$, we have
	\begin{align*}
		\mathbf{Id}_{H_1 \oplus H_2} - \mathbf{B^*}\mathbf{B} & = \begin{bmatrix}
			\mathrm{Id}_{H_1}-C^* C & -C^* D \\
			-D^*C & \mathrm{Id}_{H_2}-D^*D
		\end{bmatrix} \\
		& = \begin{bmatrix}
			\id_{H_1}-C^* C & -C^*D_{C^*}Y_0 \\
			-Y_0^*D_{C^*}C & \id_{H_2}-Y_0^*D_{C^*}D_{C^*}Y_0
		\end{bmatrix} \\
		& = \begin{bmatrix}
			\id_{H_1}-C^* C & -D_CC^*Y_0 \\
			-Y_0^*CD_C & \id_{H_2}-Y_0^*Y_0+Y_0^*CC^*Y_0
		\end{bmatrix} \\
		& = \mathbf{S^* S},
	\end{align*}
 where $\mathbf{S} = \begin{bmatrix}
					D_C  & -C^*Y_0 \\
					0 & D_{Y_0}
			\end{bmatrix}$.
			
   For every $w \in H_1 \oplus H_2$, we have 
   $$\ps{(\mathbf{Id}_{H_1 \oplus H_2} - \mathbf{B^*}\mathbf{B})w,w} = \ps{\mathbf{S^*}\mathbf{S}w,w},$$ 
   which is equivalent to $\|\mathbf{S}w\| = \left\|D_{\mathbf{B}}w\right\|$. Thus, there is an isometry $\mathbf{U} \in \mathcal{B}(H_1 \oplus H_2)$ such that $\mathbf{S}=\mathbf{U}D_{\mathbf{B}}$. Indeed, let $U : \text{Im}\left(D_{\mathbf{B}}\right) \to H_1 \oplus H_2$, $D_{\mathbf{B}} x \mapsto \mathbf{S}x$. We extend $\mathbf{U}$ by continuity to $\cj{\text{Im}\left(D_{\mathbf{B}}\right)}$, and we set $\mathbf{U}= \mathrm{Id}$ on $\text{Im}\left(D_{\mathbf{B}}\right)^{\perp}$.

 Suppose that $T$ is a contraction. Then, by Corollary~\ref{parrott-col-iso}, there exists a contraction \[\mathbf{V}=\begin{bmatrix} V_1 & V_2 \end{bmatrix} \in \mathcal{B}(H_1 \oplus H_2, K_1)\] such that
		\begin{empheq}[left=\empheqlbrace]{alignat=2}
			& \mathbf{A} = \mathbf{V} \mathbf{U} D_{\mathbf{B}} \label{A=VUD_B}, \\
			& \mathbf{V} = 0 \text{ on } \text{Im}(\mathbf{S})^{\perp}. \label{v=0orth}
		\end{empheq}
		
		By Corollary~\ref{parrott-vect-ligne}, there exists a contraction $W \in \mathcal{B}(H_2,K_1)$ such that $\mathbf{V} = \begin{bmatrix}
			V_1 & D_{V_1^*}W
		\end{bmatrix}$. The operator $W$ can be chosen such that $\text{Im}(D_{V_1^*}) \subset \text{Ker}(W^*)$ (in that case, the minimal solution $W_0$ is unique).
		
		Then, \eqref{A=VUD_B} is equivalent to
		\begin{align}
			\begin{bmatrix} A & B \end{bmatrix}  & = \begin{bmatrix}
				V_1 & D_{V_1^*}W \end{bmatrix} \begin{bmatrix}D_C & -C^*Y_0 \\ 0 & D_{Y_0} \end{bmatrix} \nonumber \\
			& = \begin{bmatrix} V_1 D_C & -V_1C^*Y_0 + D_{V_1^*}WD_{Y_0} \end{bmatrix}. \label{penultieme_etape_parrott}
		\end{align}
		
		In particular, we have $A=V_1D_C$. We now show that $V_1=Z_0$.
		
		\setcounter{fait}{0}
		\begin{fait} \label{fait}
			$\text{Im}(D_C)^{\perp} \oplus \{0\} \subset \text{Im}(\mathbf{S})^{\perp}$.
		\end{fait}
		
  \begin{proof}
			Let $v \in \text{Im}(D_C)^{\perp} = \text{Ker}(D_C)$. 
			In order to prove that $\begin{bmatrix}
				v \\ 0
			\end{bmatrix} \in \text{Ker}(\mathbf{S^*}) = \text{Im}(\mathbf{S})^{\perp}$, notice that we have 
			\begin{align*}
				\mathbf{S^*}\begin{bmatrix}
					v \\ 0
				\end{bmatrix} = \begin{bmatrix}
					0 \\ -Y_0^*Cv
				\end{bmatrix}.
			\end{align*}
			As we know that $Y_0^*=0$ on $\text{Im}(D_{C^*})^{\perp} = \text{Ker}(D_{C^*})$, it is enough to show $Cv \in \text{Ker}(D_{C^*})$. Using again the identity $CD_C = D_{C^*}C$, we have
			\begin{align*}
				\|D_{C^*}Cv\|^2 = \ps{D_{C^*}Cv, D_{C^*}Cv} = \ps{Cv, D_{C^*}^2Cv} = \ps{Cv, C D_C^2v} = 0 ,
			\end{align*}
   which completes the proof of the Fact~\ref{fait}.
		\end{proof}
		Continuing the proof of Theorem \ref{parrott}, we can deduce from \eqref{v=0orth} that $V_1=0$ on $\text{Im}(D_C)^{\perp}$ and, thus, $V_1=Z_0$. Finally, \eqref{penultieme_etape_parrott} is equivalent to 
				\begin{equation*}
			 B = -Z_0C^*Y_0 + D_{Z_0^*}WD_{Y_0}.
		\end{equation*}
		
		 Conversely, if there exists a contraction $W \in \mathcal{B}(H_2,K_1)$ such that $B=D_{Z_0^*}VD_{Y_0} - Z_0C^*Y_0$, then it is easy to check that $\mathbf{A} = \mathbf{V'} \mathbf{S} = \mathbf{V'} \mathbf{U}\mathbf{D_B}$, with $\mathbf{V'}=\begin{bmatrix}
			Z & D_{Z^*}W
		\end{bmatrix}$. As $\mathbf{V'}$ is a contraction (Corollary~\ref{parrott-vect-ligne}), this implies that $T$ is a contraction (Corollary~\ref{parrott-col-iso}).
\end{proof}

\backmatter

\printbibliography

@article {Abate,
    AUTHOR = {Abate, Marco},
     TITLE = {Multipoint {J}ulia theorems},
   JOURNAL = {Atti Accad. Naz. Lincei Rend. Lincei Mat. Appl.},
  FJOURNAL = {Atti della Accademia Nazionale dei Lincei. Rendiconti Lincei.
              Matematica e Applicazioni},
    VOLUME = {32},
      YEAR = {2021},
    NUMBER = {3},
     PAGES = {593--625},
      %ISSN = {1120-6330},
   MRCLASS = {30F45 (30C80 30D05 30E25 30J10)},
  MRNUMBER = {4355407},
MRREVIEWER = {D. Drasin},
       DOI = {10.4171/rlm/950},
       %URL = {https://doi.org/10.4171/rlm/950},
}

@article {Agler:Invent,
    AUTHOR = {Agler, Jim},
     TITLE = {Operator theory and the {C}arath\'{e}odory metric},
   JOURNAL = {Invent. Math.},
  FJOURNAL = {Inventiones Mathematicae},
    VOLUME = {101},
      YEAR = {1990},
    NUMBER = {2},
     PAGES = {483--500},
      %ISSN = {0020-9910,1432-1297},
   MRCLASS = {47A20 (32H15)},
  MRNUMBER = {1062972},
MRREVIEWER = {Joseph\ A.\ Ball},
       DOI = {10.1007/BF01231512},
       %URL = {https://doi.org/10.1007/BF01231512},
}

@article {AM:acta,
    AUTHOR = {Agler, Jim and McCarthy, John E.},
     TITLE = {Distinguished varieties},
   JOURNAL = {Acta Math.},
  FJOURNAL = {Acta Mathematica},
    VOLUME = {194},
      YEAR = {2005},
    NUMBER = {2},
     PAGES = {133--153},
      %ISSN = {0001-5962,1871-2509},
   MRCLASS = {47A13 (32C25 47A20 47A48 47A57)},
  MRNUMBER = {2231339},
MRREVIEWER = {Dmitry\ Kaliuzhnyi-Verbovetskyi},
       DOI = {10.1007/BF02393219},
       %URL = {https://doi.org/10.1007/BF02393219},
}

@book{aglerOperatorAnalysisHilbert2020,
  title = {Operator Analysis---{{Hilbert}} Space Methods in Complex Analysis},
  author = {Agler, Jim and McCarthy, John Edward and Young, Nicholas},
  date = {2020},
  series = {Cambridge {{Tracts}} in {{Mathematics}}},
  volume = {219},
  publisher = {Cambridge University Press, Cambridge},
  doi = {10.1017/9781108751292},
  %isbn = {978-1-108-48544-9},
  mrnumber = {4411370},
  pagetotal = {xv+375}
}

@article {ADR,
    AUTHOR = {Anderson, J. Milne and Dritschel, Michael A. and Rovnyak,
              James},
     TITLE = {Schwarz-{P}ick inequalities for the {S}chur-{A}gler class on
              the polydisk and unit ball},
   JOURNAL = {Comput. Methods Funct. Theory},
  FJOURNAL = {Computational Methods and Function Theory},
    VOLUME = {8},
      YEAR = {2008},
    NUMBER = {1-2},
     PAGES = {339--361},
      %ISSN = {1617-9447},
   MRCLASS = {47B32 (30C80 32A70 47A48 47N70)},
  MRNUMBER = {2419482},
MRREVIEWER = {Joseph\ A.\ Ball},
       DOI = {10.1007/BF03321692},
       %URL = {https://doi.org/10.1007/BF03321692},
}

@article {AR,
    AUTHOR = {Anderson, J. M. and Rovnyak, J.},
     TITLE = {On generalized {S}chwarz-{P}ick estimates},
   JOURNAL = {Mathematika},
  FJOURNAL = {Mathematika. A Journal of Pure and Applied Mathematics},
    VOLUME = {53},
      YEAR = {2006},
    NUMBER = {1},
     PAGES = {161--168},
      %ISSN = {0025-5793},
   MRCLASS = {30C80 (30E05 46E22)},
  MRNUMBER = {2304058},
MRREVIEWER = {Aristomenis\ Siskakis},
       DOI = {10.1112/S0025579300000085},
       %URL = {https://doi.org/10.1112/S0025579300000085},
}

@article {ArseneGheondea,
    AUTHOR = {Arsene, Gr. and Gheondea, A.},
     TITLE = {Completing matrix contractions},
   JOURNAL = {J. Operator Theory},
  FJOURNAL = {Journal of Operator Theory},
    VOLUME = {7},
      YEAR = {1982},
    NUMBER = {1},
     PAGES = {179--189},
      %ISSN = {0379-4024},
   MRCLASS = {47A20},
  MRNUMBER = {650203},
MRREVIEWER = {Heinz\ Langer},
}

@article{baribeauHyperbolicDividedDifferences2009,
  title = {On Hyperbolic Divided Differences and the {{Nevanlinna-Pick}} Problem},
  author = {Baribeau, Line and Rivard, Patrice and Wegert, Elias},
  date = {2009},
  journaltitle = {Computational Methods and Function Theory},
  shortjournal = {Comput. Methods Funct. Theory},
  volume = {9},
  number = {2},
  pages = {391--405},
  %issn = {1617-9447},
  doi = {10.1007/BF03321735},
  mrnumber = {2572646}
}

@article{beardonMultipointSchwarzPickLemma2004,
  title = {A Multi-Point {{Schwarz-Pick}} Lemma},
  author = {Beardon, A. F. and Minda, D.},
  date = {2004},
  journaltitle = {Journal d'Analyse Math\'ematique},
  shortjournal = {J. Anal. Math.},
  volume = {92},
  pages = {81--104},
  %issn = {0021-7670,1565-8538},
  doi = {10.1007/BF02787757},
  mrnumber = {2072742}
}

@article{bhatiaHowWhySolve1997,
  title = {How and Why to Solve the Operator Equation {{AX-XB}}={{Y}}},
  author = {Bhatia, Rajendra and Rosenthal, Peter},
  date = {1997},
  journaltitle = {The Bulletin of the London Mathematical Society},
  shortjournal = {Bull. London Math. Soc.},
  volume = {29},
  number = {1},
  pages = {1--21},
  %issn = {0024-6093,1469-2120},
  doi = {10.1112/S0024609396001828},
  mrnumber = {1416400}
}

@article{bickelCrouzeixConjectureRelated2020,
  title = {Crouzeix's Conjecture and Related Problems},
  author = {Bickel, Kelly and Gorkin, Pamela and Greenbaum, Anne and Ransford, Thomas and Schwenninger, Felix L. and Wegert, Elias},
  date = {2020},
  journaltitle = {Computational Methods and Function Theory},
  shortjournal = {Comput. Methods Funct. Theory},
  volume = {20},
  number = {3-4},
  pages = {701--728},
  %issn = {1617-9447,2195-3724},
  doi = {10.1007/s40315-020-00350-9},
  mrnumber = {4175507}
}

@article{choMultipointSchwarzPickLemma2012,
  title = {On a Multi-Point {{Schwarz-Pick}} Lemma},
  author = {Cho, Kyung Hyun and Kim, Seong-A and Sugawa, Toshiyuki},
  date = {2012},
  journaltitle = {Computational Methods and Function Theory},
  shortjournal = {Comput. Methods Funct. Theory},
  volume = {12},
  number = {2},
  pages = {483--499},
  %issn = {1617-9447},
  doi = {10.1007/BF03321839},
  mrnumber = {3058518}
}

@book{constantinescuSchurParametersFactorization1996,
  title = {Schur Parameters, Factorization and Dilation Problems},
  author = {Constantinescu, Tiberiu},
  date = {1996},
  series = {Operator {{Theory}}: {{Advances}} and {{Applications}}},
  volume = {82},
  publisher = {Birkh\"auser Verlag, Basel},
  %isbn = {978-3-7643-5285-1},
  mrnumber = {1399080},
  pagetotal = {x+253}
}

@article{davisNormpreservingDilationsTheir1982,
  title = {Norm-Preserving Dilations and Their Applications to Optimal Error Bounds},
  author = {Davis, Chandler and Kahan, W. M. and Weinberger, H. F.},
  date = {1982},
  journaltitle = {SIAM Journal on Numerical Analysis},
  shortjournal = {SIAM J. Numer. Anal.},
  volume = {19},
  number = {3},
  pages = {445--469},
  %issn = {0036-1429},
  doi = {10.1137/0719029},
  mrnumber = {656462}
}

@article{douglasMajorizationFactorizationRange1966,
  title = {On Majorization, Factorization, and Range Inclusion of Operators on {{Hilbert}} Space},
  author = {Douglas, R. G.},
  date = {1966},
  journaltitle = {Proceedings of the American Mathematical Society},
  shortjournal = {Proc. Amer. Math. Soc.},
  volume = {17},
  pages = {413--415},
  %issn = {0002-9939,1088-6826},
  doi = {10.2307/2035178},
  mrnumber = {203464}
}

@incollection{druryRemarksNeumannInequality1983,
  title = {Remarks on von {{Neumann}}'s Inequality},
  booktitle = {Banach Spaces, Harmonic Analysis, and Probability Theory ({{Storrs}}, {{Conn}}., 1980/1981)},
  author = {Drury, S. W.},
  date = {1983},
  series = {Lecture {{Notes}} in {{Math}}.},
  volume = {995},
  pages = {14--32},
  publisher = {Springer, Berlin},
  doi = {10.1007/BFb0061886},
  %isbn = {978-3-540-12314-9},
  mrnumber = {717226}
}

@incollection {history,
    AUTHOR = {Elin, Mark and Jacobzon, Fiana and Levenshtein, Marina and
              Shoikhet, David},
     TITLE = {The {S}chwarz lemma: rigidity and dynamics},
 BOOKTITLE = {Harmonic and complex analysis and its applications},
    SERIES = {Trends Math.},
     PAGES = {135--230},
 PUBLISHER = {Birkh\"{a}user/Springer, Cham},
      YEAR = {2014},
      %ISBN = {978-3-319-01805-8; 978-3-319-01806-5},
   MRCLASS = {30C80},
  MRNUMBER = {3203101},
MRREVIEWER = {Stephen\ M.\ Zemyan},
       DOI = {10.1007/978-3-319-01806-5\_3},
       %URL = {https://doi.org/10.1007/978-3-319-01806-5_3},
}

@article{FanAnalyticfunctionsproper1978,
  title = {Analytic Functions of a Proper Contraction},
  author = {Fan, Ky},
  date = {1978},
  journaltitle = {Mathematische Zeitschrift},
  shortjournal = {Math Z},
  volume = {160},
  number = {3},
  pages = {275--290},
  %issn = {1432-1823},
  doi = {10.1007/BF01237041},
  langid = {english},
  keywords = {Analytic Function,Proper Contraction}
}

@book {FoiasFrazho,
    AUTHOR = {Foias, Ciprian and Frazho, Arthur E.},
     TITLE = {The commutant lifting approach to interpolation problems},
    SERIES = {Operator Theory: Advances and Applications},
    VOLUME = {44},
 PUBLISHER = {Birkh\"{a}user Verlag, Basel},
      YEAR = {1990},
     PAGES = {xxiv+632},
      %ISBN = {3-7643-2461-9},
   MRCLASS = {47A57 (30E05 47-02 47A20 47N70 93B28)},
  MRNUMBER = {1120546},
       DOI = {10.1007/978-3-0348-7712-1},
       %URL = {https://doi.org/10.1007/978-3-0348-7712-1},
}

@book{garciaIntroductionModelSpaces2016,
  title = {Introduction to Model Spaces and Their Operators},
  author = {Garcia, Stephan Ramon and Mashreghi, Javad and Ross, William T.},
  date = {2016},
  series = {Cambridge {{Studies}} in {{Advanced Mathematics}}},
  volume = {148},
  publisher = {Cambridge University Press, Cambridge},
  doi = {10.1017/CBO9781316258231},
  %isbn = {978-1-107-10874-5},
  mrnumber = {3526203},
  pagetotal = {xv+322}
}

@thesis{gupta,
  type = {phdthesis},
  title = {The Carath\'eodory-F\'ejer Interpolation Problems and the von-{{Neumann}} Inequality},
  author = {Gupta, Rajeev},
  date = {2015},
  institution = {Department of Mathematics, Indian Institute of Science, Bangalore}
}

@article{jocicNoncommutativeSchwarzLemma2022,
  title = {Noncommutative {{Schwarz}} Lemma and {{Pick-Julia}} Theorems for Generalized Derivations in Norm Ideals of Compact Operators},
  author = {Joci\'c, Danko R.},
  date = {2022},
  journaltitle = {Complex Analysis and Operator Theory},
  shortjournal = {Complex Anal. Oper. Theory},
  volume = {16},
  number = {8},
  pages = {Paper No. 111, 23},
  %issn = {1661-8254,1661-8262},
  doi = {10.1007/s11785-022-01287-8},
  mrnumber = {4502819}
}

@article {JKM,
    AUTHOR = {Jury, Michael T. and Knese, Greg and McCullough, Scott},
     TITLE = {Nevanlinna-{P}ick interpolation on distinguished varieties in
              the bidisk},
   JOURNAL = {J. Funct. Anal.},
  FJOURNAL = {Journal of Functional Analysis},
    VOLUME = {262},
      YEAR = {2012},
    NUMBER = {9},
     PAGES = {3812--3838},
      %ISSN = {0022-1236,1096-0783},
   MRCLASS = {47A57 (30E05 32A36 32E30 47A48)},
  MRNUMBER = {2899979},
MRREVIEWER = {Sanne\ ter Horst},
       DOI = {10.1016/j.jfa.2012.01.028},
       %URL = {https://doi.org/10.1016/j.jfa.2012.01.028},
}

@article{kimInvariantDifferentialOperators2007,
  title = {Invariant Differential Operators Associated with a Conformal Metric},
  author = {Kim, Seong-A and Sugawa, Toshiyuki},
  date = {2007},
  journaltitle = {Michigan Mathematical Journal},
  shortjournal = {Michigan Math. J.},
  volume = {55},
  number = {2},
  pages = {459--479},
  %issn = {0026-2285,1945-2365},
  doi = {10.1307/mmj/1187647003},
  mrnumber = {2369945}
}

@article {kneseSchwarzLemmaPolydisk2007a,
    AUTHOR = {Knese, Greg},
     TITLE = {A {S}chwarz lemma on the polydisk},
   JOURNAL = {Proc. Amer. Math. Soc.},
  FJOURNAL = {Proceedings of the American Mathematical Society},
    VOLUME = {135},
      YEAR = {2007},
    NUMBER = {9},
     PAGES = {2759--2768},
      %ISSN = {0002-9939,1088-6826},
   MRCLASS = {32A30 (30C80 47A57)},
  MRNUMBER = {2317950},
MRREVIEWER = {Jos\'{e}\ Pastor\ Gim\'{e}nez},
       DOI = {10.1090/S0002-9939-07-08766-7},
       %URL = {https://doi.org/10.1090/S0002-9939-07-08766-7},
}

@article{kneseNeumannInequalityTimes2016,
  author = {Knese, Greg},
  date = {2016},
  journaltitle = {Bulletin of the London Mathematical Society},
  shortjournal = {Bull. Lond. Math. Soc.},
  volume = {48},
  number = {1},
  pages = {53--57},
  %issn = {0024-6093,1469-2120},
  doi = {10.1112/blms/bdv087},
  mrnumber = {3455747},
  title = {The von {{Neumann}} Inequality for {$3\times 3$} Matrices}
}

@article {McC1,
    AUTHOR = {MacCluer, Barbara D. and Stroethoff, Karel and Zhao, Ruhan},
     TITLE = {Generalized {S}chwarz-{P}ick estimates},
   JOURNAL = {Proc. Amer. Math. Soc.},
  FJOURNAL = {Proceedings of the American Mathematical Society},
    VOLUME = {131},
      YEAR = {2003},
    NUMBER = {2},
     PAGES = {593--599},
      %ISSN = {0002-9939,1088-6826},
   MRCLASS = {30C80 (30D45 47B38)},
  MRNUMBER = {1933351},
MRREVIEWER = {Yuri\ A.\ Farkov},
       DOI = {10.1090/S0002-9939-02-06588-7},
       %URL = {https://doi.org/10.1090/S0002-9939-02-06588-7},
}

@book {Nikolski,
    AUTHOR = {Nikolski, N. K.},
     TITLE = {Treatise on the shift operator},
    SERIES = {Grundlehren der mathematischen Wissenschaften [Fundamental
              Principles of Mathematical Sciences]},
    VOLUME = {273},
      NOTE = {Spectral function theory,
              With an appendix by S. V. Hru\v{s}\v{c}ev [S. V.
              Khrushch\"{e}v] and V. V. Peller,
              Translated from the Russian by Jaak Peetre},
 PUBLISHER = {Springer-Verlag, Berlin},
      YEAR = {1986},
     PAGES = {xii+491},
      %ISBN = {3-540-15021-8},
   MRCLASS = {47B37 (30H05 47-02 47B35)},
  MRNUMBER = {827223},
MRREVIEWER = {James\ Rovnyak},
       DOI = {10.1007/978-3-642-70151-1},
       %URL = {https://doi.org/10.1007/978-3-642-70151-1},
}

@incollection {NikVas,
    AUTHOR = {Nikolski, Nikolai and Vasyunin, Vasily},
     TITLE = {Elements of spectral theory in terms of the free function
              model. {I}. {B}asic constructions},
 BOOKTITLE = {Holomorphic spaces ({B}erkeley, {CA}, 1995)},
    SERIES = {Math. Sci. Res. Inst. Publ.},
    VOLUME = {33},
     PAGES = {211--302},
 PUBLISHER = {Cambridge Univ. Press, Cambridge},
      YEAR = {1998},
      ISBN = {0-521-63193-9},
   MRCLASS = {47-02 (47A10 47A15 47A45)},
  MRNUMBER = {1630652},
}

@article{parrottQuotientNormSz1978,
  title = {On a quotient norm and the {S}z.-{N}agy-
              {F}oia\c{s} lifting theorem},
  author = {Parrott, Stephen},
  date = {1978},
  journaltitle = {Journal of Functional Analysis},
  shortjournal = {J. Functional Analysis},
  volume = {30},
  number = {3},
  pages = {311--328},
  %issn = {0022-1236},
  doi = {10.1016/0022-1236(78)90060-5},
  mrnumber = {518338}
}

@article{paulsenBohrInequality2002,
  title = {On {{Bohr}}'s Inequality},
  author = {Paulsen, Vern I. and Popescu, Gelu and Singh, Dinesh},
  date = {2002},
  journaltitle = {Proceedings of the London Mathematical Society. Third Series},
  shortjournal = {Proc. London Math. Soc. (3)},
  volume = {85},
  number = {2},
  pages = {493--512},
  %issn = {0024-6115,1460-244X},
  doi = {10.1112/S0024611502013692},
  mrnumber = {1912059}
}

@book{paulsenCompletelyBoundedMaps2002,
  title = {Completely Bounded Maps and Operator Algebras},
  author = {Paulsen, Vern},
  date = {2002},
  series = {Cambridge {{Studies}} in {{Advanced Mathematics}}},
  volume = {78},
  publisher = {Cambridge University Press, Cambridge},
  mrnumber = {1976867},
  pagetotal = {xii+300}
}

@article {PtakLAA,
    AUTHOR = {Pt\'{a}k, Vlastimil},
     TITLE = {A maximum problem for matrices},
   JOURNAL = {Linear Algebra Appl.},
  FJOURNAL = {Linear Algebra and its Applications},
    VOLUME = {28},
      YEAR = {1979},
     PAGES = {193--204},
      %ISSN = {0024-3795,1873-1856},
   MRCLASS = {15A60 (15A45 47A30)},
  MRNUMBER = {549433},
MRREVIEWER = {M.\ F.\ Smiley},
       DOI = {10.1016/0024-3795(79)90132-0},
       %URL = {https://doi.org/10.1016/0024-3795(79)90132-0},
}

@article {PtakYoung,
    AUTHOR = {Pt\'{a}k, Vlastimil and Young, N. J.},
     TITLE = {Functions of operators and the spectral radius},
   JOURNAL = {Linear Algebra Appl.},
  FJOURNAL = {Linear Algebra and its Applications},
    VOLUME = {29},
      YEAR = {1980},
     PAGES = {357--392},
   MRCLASS = {47A30 (15A60 47A60)},
  MRNUMBER = {562769},
MRREVIEWER = {Lawrence\ R.\ Williams},
       DOI = {10.1016/0024-3795(80)90250-5},
       %URL = {https://doi.org/10.1016/0024-3795(80)90250-5},
}

@article {RivardPAMS,
    AUTHOR = {Rivard, Patrice},
     TITLE = {A {S}chwarz-{P}ick theorem for higher-order hyperbolic
              derivatives},
   JOURNAL = {Proc. Amer. Math. Soc.},
  FJOURNAL = {Proceedings of the American Mathematical Society},
    VOLUME = {139},
      YEAR = {2011},
    NUMBER = {1},
     PAGES = {209--217},
      %ISSN = {0002-9939},
   MRCLASS = {30F45 (30C80)},
  MRNUMBER = {2729084},
MRREVIEWER = {William Ma},
       DOI = {10.1090/S0002-9939-2010-10488-4},
       %URL = {https://doi.org/10.1090/S0002-9939-2010-10488-4},
}

@article {RivardCAOT,
    AUTHOR = {Rivard, Patrice},
     TITLE = {Some applications of higher-order hyperbolic derivatives},
   JOURNAL = {Complex Anal. Oper. Theory},
  FJOURNAL = {Complex Analysis and Operator Theory},
    VOLUME = {7},
      YEAR = {2013},
    NUMBER = {4},
     PAGES = {1127--1156},
      %ISSN = {1661-8254},
   MRCLASS = {30F45 (30C80 30E10)},
  MRNUMBER = {3079846},
MRREVIEWER = {William Ma},
       DOI = {10.1007/s11785-011-0172-z},
       %URL = {https://doi.org/10.1007/s11785-011-0172-z},
}

@book{rosenblumHardyClassesOperator1997,
  title = {Hardy Classes and Operator Theory},
  author = {Rosenblum, Marvin and Rovnyak, James},
  date = {1997},
  publisher = {Dover Publications, Inc., Mineola, NY},
  mrnumber = {1435287},
  pagetotal = {xiv+161}
}

@article{rosenblumOperatorEquationBXXA1956,
  title = {On the Operator Equation {{BX-XA}}={{Q}}},
  author = {Rosenblum, Marvin},
  date = {1956},
  journaltitle = {Duke Mathematical Journal},
  shortjournal = {Duke Math. J.},
  volume = {23},
  pages = {263--269},
  mrnumber = {79235}
}

@book{rudinFunctionTheoryPolydiscs1969,
  title = {Function Theory in Polydiscs},
  author = {Rudin, Walter},
  date = {1969},
  publisher = {W. A. Benjamin, Inc., New York-Amsterdam},
  mrnumber = {255841},
  pagetotal = {vii+188}
}

@article {ruscheweyhSerdica,
    AUTHOR = {Ruscheweyh, St.},
     TITLE = {Two remarks on bounded analytic functions},
   JOURNAL = {Serdica},
  FJOURNAL = {Serdica. Bulgaricae Mathematicae Publicationes},
    VOLUME = {11},
      YEAR = {1985},
    NUMBER = {2},
     PAGES = {200--202},
   MRCLASS = {30B10},
  MRNUMBER = {817957},
MRREVIEWER = {G.\ P.\ Bhargava},
}

@article {sarasonGeneralizedInterpolationSpinfty1967,
    AUTHOR = {Sarason, Donald},
     TITLE = {Generalized interpolation in {$H\sp{\infty }$}},
   JOURNAL = {Trans. Amer. Math. Soc.},
  FJOURNAL = {Transactions of the American Mathematical Society},
    VOLUME = {127},
      YEAR = {1967},
     PAGES = {179--203},
   MRCLASS = {46.55 (30.00)},
  MRNUMBER = {208383},
MRREVIEWER = {Ronald\ G.\ Douglas},
       DOI = {10.2307/1994641},
       %URL = {https://doi.org/10.2307/1994641},
}

@book {Simon,
    AUTHOR = {Simon, Barry},
     TITLE = {Loewner's the\-o\-rem on monotone matrix functions},
    SERIES = {Grundlehren der math\-e\-ma\-tischen Wis\-senschaften [Fundamental
              Principles of Mathematical Sciences]},
    VOLUME = {354},
 PUBLISHER = {Springer, Cham},
      YEAR = {2019},
     PAGES = {xi+459},
      %ISBN = {978-3-030-22421-9; 978-3-030-22422-6},
   MRCLASS = {47A63 (26A48 30B40 30E05 30E10 30J10)},
  MRNUMBER = {3969971},
MRREVIEWER = {Linda\ J.\ Patton},
       DOI = {10.1007/978-3-030-22422-6},
       %URL = {https://doi.org/10.1007/978-3-030-22422-6},
}

@article {Szehr,
    AUTHOR = {Szehr, Oleg},
     TITLE = {Eigenvalue estimates for the resolvent of a non-normal matrix},
   JOURNAL = {J. Spectr. Theory},
  FJOURNAL = {Journal of Spectral Theory},
    VOLUME = {4},
      YEAR = {2014},
    NUMBER = {4},
     PAGES = {783--813},
   MRCLASS = {15A42 (65N35)},
  MRNUMBER = {3299814},
MRREVIEWER = {Li\ Min\ Zou},
       DOI = {10.4171/JST/86},
       %URL = {https://doi.org/10.4171/JST/86},
}

@article{yamashitaPickVersionSchwarz1994,
  title = {The {{Pick}} Version of the {{Schwarz}} Lemma and Comparison of the {{Poincar\'e}} Densities},
  author = {Yamashita, Shinji},
  date = {1994},
  journaltitle = {Annales Academiae Scientiarum Fennicae. Series A I. Mathematica},
  shortjournal = {Ann. Acad. Sci. Fenn. Ser. A I Math.},
  volume = {19},
  number = {2},
  pages = {291--322},
  mrnumber = {1274084}
}

@article {Young2,
    AUTHOR = {Young, N. J.},
     TITLE = {Analytic programmes in matrix algebras},
   JOURNAL = {Proc. London Math. Soc. (3)},
  FJOURNAL = {Proceedings of the London Mathematical Society. Third Series},
    VOLUME = {36},
      YEAR = {1978},
    NUMBER = {2},
     PAGES = {226--242},
   MRCLASS = {15A60 (47A30)},
  MRNUMBER = {484965},
MRREVIEWER = {V.\ Pt\'{a}k},
DOI = {10.1112/plms/s3-36.2.226},
       %URL = {https://doi.org/10.1112/plms/s3-36.2.226},
}

@book{youngIntroductionHilbertSpace1988,
  title = {An Introduction to {{Hilbert}} Space},
  author = {Young, Nicholas},
  date = {1988},
  series = {Cambridge {{Mathematical Textbooks}}},
  publisher = {Cambridge University Press, Cambridge},
  doi = {10.1017/CBO9781139172011},
  mrnumber = {949693},
  pagetotal = {x+239}
}
			
\end{document}